\documentclass[12pt,draftcls,onecolumn]{IEEEtran/IEEEtran}   
\usepackage{cite}

\ifCLASSINFOpdf
  \usepackage[pdftex]{graphicx}
  \graphicspath{{graphics/}}
  \DeclareGraphicsExtensions{.pdf,.jpeg,.png}
\else
  \usepackage[dvips]{graphicx}
  \graphicspath{{graphics/}}
  \DeclareGraphicsExtensions{.eps}
\fi
\usepackage[cmex10]{amsmath}
\usepackage{ifthen}
\newboolean{TwoColEq}
\setboolean{TwoColEq}{false} 

\usepackage{amssymb}  
\usepackage{amsthm}
\usepackage{algorithm} 
\usepackage{color}
\usepackage{enumerate}
\usepackage[textwidth=2cm,textsize=small]{todonotes}

\theoremstyle{plain}
\newtheorem{thm}{Theorem}

\newtheorem{lem}{Lemma}
\newtheorem{cor}{Corollary}

\theoremstyle{definition}
\newtheorem{definition}{Definition}
\newtheorem{exmp}{Example}

\theoremstyle{remark}

\newtheorem{remark}{Remark}

\newcommand{\be}{\begin{enumerate}}
\newcommand{\ee}{\end{enumerate}}

\newcommand{\Tr}{\text{Tr}}

\newcommand{\TF}[4]{\left [ \begin{array}{c|c}{#1} & {#2} \\ \hline {#3} & {#4} \end{array} \right ]}

\newcommand{\D}{\mathsf D}
\newcommand{\R}{\mathsf R}
\newcommand{\Adj}{\text{Adj}}
\newcommand{\Lap}{\text{Lap}}

\newcommand{\Prob}{\mathbb P}
\newcommand{\Exp}{\mathbb E}

\begin{document}
%
\title{Differentially Private Filtering}
%
%
%

\author{Jerome~Le~Ny,~\IEEEmembership{Member,~IEEE,}
        and~George~J.~Pappas,~\IEEEmembership{Fellow,~IEEE}
\thanks{J. Le Ny is with the department of Electrical Engineering, Ecole Polytechnique de Montreal,
QC H3T 1J4, Canada. G.~Pappas is with the Department of Electrical and Systems Engineering, 
University of Pennsylvania, Philadelphia, PA 19104, USA. {\tt\small jerome.le-ny@polymtl.ca, 
pappasg@seas.upenn.edu.}}
\thanks{Preliminary versions of this paper will appear at Allerton 2012 and CDC 2012.}
}

%
%

\markboth{}
{Le Ny and Pappas: Differentially Private Filtering}
%



\maketitle

\begin{abstract}
Emerging systems such as smart grids or intelligent transportation systems 
often require end-user applications to continuously send information to external 
data aggregators performing monitoring or control tasks. This can result in 
an undesirable loss of privacy for the users in exchange of the benefits provided 
by the application.
Motivated by this trend, this paper introduces privacy concerns in a system theoretic context, 
and addresses the problem of releasing filtered signals that respect the privacy 
of the user data streams. Our approach relies on a formal notion of privacy from the 
database literature, called \emph{differential privacy}, which provides strong privacy 
guarantees against adversaries with arbitrary side information. 
Methods are developed to approximate a given filter by a differentially private 
version, so that the distortion introduced by the privacy mechanism is minimized. 
Two specific scenarios are considered. First, the notion of differential privacy is
extended to dynamic systems with many participants contributing independent 
input signals. Kalman filtering is also discussed in this context, when a released 
output signal must preserve differential privacy for the measured signals or 
state trajectories of the individual participants.
Second, differentially private mechanisms are described to approximate stable filters 
when participants contribute to a single event stream, extending previous work 
on differential privacy under continual observation.
\end{abstract}


%
\IEEEpeerreviewmaketitle

\section{Introduction}
%
%
%
%


\IEEEPARstart{A}{ rapidly} growing number of applications requires users to release 
private data streams to third-party applications for signal processing and 
decision-making purposes. Examples include smart grids, population health monitoring, 
online recommendation systems, traffic monitoring, fuel consumption optimization, 
and cloud computing for industrial control systems.
For privacy or security reasons, the participants benefiting from the services
provided by these systems generally do not want to release 
more information than strictly necessary.
In a smart grid for example, a customer could receive 
better rates in exchange of continuously sending to the utility company
her instantaneous power consumption, thereby helping to improve the 
demand forecast mechanism. In doing so however, she is also informing the utility or 
a potential eavesdropper about the type of appliances she owns 
as well as her daily activities \cite{Hart92_loadMonitoring}.
Similarly, individual private signals can be recovered from published outputs 
aggregated from many users, and anonymizing a dataset is not enough
to guarantee privacy, due to the existence of public side information.
This is demonstrated in \cite{Narayanan08_netflixBreach, Calandrino11_privacyAttackCollabFilt} 
for example, where private ratings and transactions from individuals on commercial websites are 
successfully inferred with the help of information from public recommendation systems. 
Emerging traffic monitoring systems using position measurements from 
smartphones \cite{Hoh11_VTL_trafficMonitoring}
is another application area where individual position traces can be re-identified 
by correlating them with public information such as a person's location of residence 
or work \cite{Hoh11_VTL_trafficMonitoring}.
Hence the development of rigorous privacy preserving mechanisms is crucial 
to address the justified concerns of potential users and thus encourage an increasing level 
of participation, which can in turn greatly improve the efficiency of these large-scale systems.

Precisely defining what constitutes a breach of privacy is a delicate task.
A particularly successful recent definition of privacy used in the database literature is that of
\emph{differential privacy} \cite{Dwork06_DPcalibration}, which is
motivated by the fact that any useful information provided by a dataset 
about a group of people can compromise the privacy of specific individuals
due to the existence of side information. 
Differentially private mechanisms randomize their responses to 
dataset analysis requests and guarantee that whether or not an individual
chooses to contribute her data only marginally changes the distribution
over the published outputs. As a result, even an adversary cross-correlating
these outputs with other sources of information cannot infer much more
about specific individuals after publication than before \cite{Kasiviswanathan08_DP_sideInfo}.

Most work related to privacy is concerned with the analysis of static databases 
\cite{Dwork06_DPcalibration, Dwork06_DPgaussian, Roth10_DP_thesis, Li10_DPmatrix1}, 
whereas cyber-physical systems clearly emphasize the need for mechanisms working with dynamic, 
time-varying data streams. Recently, the problem of releasing differentially private statistics when
the input data takes the form of a binary stream has been considered in \cite{Dwork10_DPcounter, Chan11_DPcontinuous, Bolot11_DPdecayingSums}. This work is discussed in more details in
Section \ref{section: related work}. A differentially private version of the iterative averaging algorithm
for consensus is considered in \cite{Huang12_DPconsensus}. In this case, the input data to protect
consists of the initial values of the participants and is thus a single vector, but the update mechanism 
subject to privacy attacks is dynamic. Information-theoretic approaches have also been proposed 
to guarantee some level of privacy when releasing time series \cite{Varodayan11_batteryPrivacy,
Sankar11_privacyInfoTheoretic}. However, the resulting privacy guarantees only hold if the statistics
of the participants' data streams obey the assumptions made (typically stationarity, dependence and 
distributional assumptions), and require the explicit statistical modeling of all available side information.
This task is very difficult in general as new, as-yet-unknown side information can become available after 
releasing the results. In contrast, differential privacy is a worst-case notion that holds independently of 
any probabilistic assumption on the dataset, and controls the information leakage against adversaries 
with arbitrary side information \cite{Kasiviswanathan08_DP_sideInfo}. Once such a privacy guarantee 
is enforced, one can still leverage potential additional statistical information about the dataset 
to improve the quality of the outputs. 

The main contribution of this paper is to introduce privacy concerns in the context of systems theory. 
Section \ref{section: differential privacy background} provides some technical background on 
differential privacy. We then formulate in Section \ref{section: DP linear systems} the problem 
of releasing the output of a dynamical system while preserving differential privacy for the driving inputs, 
assumed to originate from different participants. It is shown that accurate results can be published for 
systems with small incremental gains with respect to the individual input channels. These results are 
extended in Section \ref{section: private KF} to the problem of designing a differentially private Kalman filter, 
as an example of situation where additional information about the process generating the individual signals 
can be leveraged to publish more accurate results.
Finally, Section \ref{section: binary streams} is motivated by the recent work on ``differential privacy under 
continual observation'' \cite{Dwork10_DPcounter, Chan11_DPcontinuous}, and considers systems processing 
a single integer-valued signal describing the occurrence of events originating from many individual participants.
Differentially private approximations of the systems are proposed with the goal of minimizing the mean 
squared error introduced by the privacy preserving mechanism. Some additional references to the related
literature are provided in Section \ref{section: related work}.


\section{Differential Privacy}	\label{section: differential privacy background}

In this section we review the notion of differential privacy \cite{Dwork06_DPcalibration} 
as well as some basic mechanisms that can be used to achieve it when the released data 
belongs to a finite-dimensional vector space.
In the original papers on differential privacy \cite{Blum05_sulq, Dwork06_DPcalibration, 
Dwork06_DPgaussian}, a sanitizing mechanism has access to a database and provides 
noisy answers to queries submitted by data analysts wishing to draw inference from the data.
However, the notion of differential privacy can be defined for fairly general types of datasets. 
Most of the results in this section are known, but in some cases we provide more precise or 
slightly different versions of some statements made in previous work. We refer the reader 
to the surveys by Dwork, e.g., \cite{Dwork_ICAL06_DP}, for additional background on 
differential privacy.

\subsection{Definition}

Let us fix some probability space $(\Omega, \mathcal F, \Prob)$. Let $\D$ be a space of 
datasets of interest (e.g., a space of data tables, or a signal space). A \emph{mechanism} 
is just a map $M: \D \times \Omega \to \R$, for some measurable output space $(\R,\mathcal M)$, 
where $\mathcal M$ denotes a $\sigma$-algebra, such that for any element $d \in \D$, $M(d,\cdot)$ 
is a random variable, typically written simply $M(d)$. A mechanism can be viewed as a probabilistic 
algorithm to answer a query $q$, which is a map $q: \D \to \R$. In some cases, we index the 
mechanism by the query $q$ of interest, writing $M_q$.

\begin{exmp}	\label{ex: space of databases}
Let $\D = \mathbb R^{n}$, with each real-valued entry of $d \in \D$ corresponding to some sensitive
information for an individual contributing her data, e.g., her salary. A data analyst would like to know 
the average of the entries of $d$, i.e., the query is $q: \D \to \mathbb R$ with 
$q(d) = \frac{1}{n} \sum_{i=1}^n d_{i}$. As detailed in Section \ref{section: basic mech}, a typical 
mechanism $M_q$ to answer this query in a differentially private way computes $q(d)$ and blurs 
the result by adding a random variable $Y: \Omega \to \mathbb R$, so that
$M_q: \D \times \Omega \to \mathbb R$ with $M_{q}(d) = \frac{1}{n} \sum_{i=1}^n d_{i} + Y$.
Note that in the absence of perturbation $Y$, an adversary who knows $n$ and all $d_j$ for $j \geq 2$ 
can recover the remaining entry $d_1$ exactly if he learns $q(d)$. This can deter people
from contributing their data, even though broader participation improves the accuracy 
of the analysis, which can provide useful knowledge to the population as a whole.
\end{exmp}

Next, we introduce the definition of differential privacy \cite{Dwork06_DPcalibration, Dwork06_DPgaussian}.  
Intuitively, in the following definition, $\D$ is a space of datasets of interest, and we have a symmetric binary 
relation $\Adj$ on $\D$, called adjacency, such that $\Adj(d,d')$ if and only if $d$ and $d'$ differ by the 
data of a single participant. 

\begin{definition}	\label{def: differential privacy original}
Let $\D$ be a space equipped with a symmetric binary relation denoted $\Adj$, and let $(\R, \mathcal M)$ 
be a measurable space. Let $\epsilon, \delta \geq 0$. A mechanism $M: \D \times \Omega \to \R$ is 
$(\epsilon, \delta)$-differentially private if for all $d,d' \in \D$ such that $\Adj(d,d')$, we have
\begin{align}	\label{eq: standard def approximate DP original}
\Prob(M(d) \in S) \leq e^{\epsilon} \Prob(M(d') \in S) + \delta, \;\; \forall S \in \mathcal M. 
\end{align}
If $\delta=0$, the mechanism is said to be $\epsilon$-differentially private. 
\end{definition}

Intuitively, this definition says that for two adjacent datasets, the distributions over the outputs of the mechanism 
should be close.  The choice of the parameters $\epsilon, \delta$ is set by the privacy policy. Typically 
$\epsilon$ is taken to be a small constant, e.g., $\epsilon \approx 0.1$ or perhaps even $\ln 2$ or $\ln 3$. 
The parameter $\delta$ should be kept small as it controls the probability of certain significant losses 
of privacy, e.g., when a zero probability event for input $d'$ becomes an event with positive probability 
for input $d$ in (\ref{eq: standard def approximate DP original}).

\begin{remark}
The definition of differential privacy depends on the choice of $\sigma$-algebra $\mathcal M$ 
in Definition \ref{def: differential privacy original}. When we need to state this $\sigma$-algebra
explicitly, we write $M: \D \times \Omega \to (\R,\mathcal M)$. 
In particular, this $\sigma$-algebra should be sufficiently large, since (\ref{eq: standard def 
approximate DP original}) is trivially satisfied by any mechanism if $\mathcal M = \{\emptyset, \R\}$. 
\end{remark}

%
%
%

The next lemma provides alternative technical characterizations of differential privacy and appears to be new. 
First, we introduce some notation. We call a signed measure $\nu$ on $(\R, \mathcal M)$ $\delta$-bounded if it satisfies 
$\nu(S) \leq \delta$ for all $S \in \mathcal M$ \cite[p.180]{Dudley02_book}. A measure is sometimes called positive measure for emphasis.
For $(\R, \mathcal M)$ a measurable space, we denote by $\mathcal F_b(\R)$ the space of bounded real-valued measurable functions 
on $\R$ and we define $\mu g := \int g \, d \mu$ for $g \in \mathcal F_b(\R)$ and $\mu$ a positive measure on $\mathcal M$.

\begin{lem}	\label{lemma: technical result equivalence DP}
The following are equivalent: 
\begin{enumerate}[(a)]
\item $M$ is $(\epsilon,\delta)$-differentially private, satisfying (\ref{eq: standard def approximate DP original}).	\label{def: DP 1}
\item For all $d,d' \in \D$ such that $\Adj(d,d')$, there exists a $\delta$-bounded positive measure $\mu^{d,d'}$ 
on $(\R,\mathcal M)$ such that  we have \label{def: DP bis}	
\begin{align} \label{eq: DP def bis}
\Prob(M(d) \in S) \leq e^{\epsilon} \Prob(M(d') \in S) + \mu^{d,d'}(S), \;\; \forall S \in \mathcal M.
\end{align}
\item For all $d,d' \in \D$ such that $\Adj(d,d')$, there exists a $\delta$-bounded positive measure $\mu^{d,d'}$ 
on $(\R,\mathcal M)$ such that for all $g \in \mathcal F_b(\R)$, we have \label{def: DP 2}	
\begin{align} \label{eq: DP def 2}
\Exp(g(M(d))) \leq e^{\epsilon} \Exp ( g(M(d')) ) + \mu^{d,d'} g.
\end{align} 
\end{enumerate}
\end{lem}

\begin{proof}
(\ref{def: DP 1}) $\Rightarrow$ (\ref{def: DP bis}). Suppose that $M$ is $(\epsilon,\delta)$-differentially private.
Define the signed measure $\nu^{d,d'}$ by $S \mapsto \nu^{d,d'}(S) := \Prob(M(d) \in S) - e^\epsilon \Prob(M(d') \in S)$
\cite[Section 5.6]{Dudley02_book}. By the definition (\ref{eq: standard def approximate DP original}), $\nu^{d,d'}$ is 
$\delta$-bounded. Let $\mu^{d,d'}$ be the positive variation of $\nu^{d,d'}$, i.e., $\mu^{d,d'}(S) = \sup\{\nu(G): G \subset S\}$, 
for all $S \in \mathcal M$. Then $\mu^{d,d'}$ is a positive measure \cite[Section 5.6]{Dudley02_book}, is $\delta$-bounded since
$\nu_{d,d'}$ is, and since $\nu^{d,d'}(S) \leq \mu^{d,d'}(S)$ for all $S \in \mathcal M$, we have (\ref{eq: DP def bis}).

(\ref{def: DP bis}) $\Rightarrow$ (\ref{def: DP 2}): 
Let $B$ be a bound on $g$. For any $k \geq 1$, we divide the interval $[-B,B]$ in $k$ 
consecutive intervals $I_i$ of length $2B/k$, and we let $A_i = g^{-1}(I_i)$ and $c_i$ be
the mid-point of the interval $I_i$. Then (\ref{def: DP 2}) holds for the simple function 
$\sum_{i=1}^k c_i 1_{A_i}$, and these functions approximate $g$. We conclude using 
the dominated convergence theorem. 

(\ref{def: DP 2}) $\Rightarrow$ (\ref{def: DP 1}): Take $g = 1_S$ and use the fact that $\mu^{d,d'}$ is $\delta$-bounded.
\end{proof}

A fundamental property of the notion of differential privacy is that no additional privacy loss can occur by simply manipulating an output 
that is differentially private. This result is similar in spirit to the data processing inequality from information theory \cite{Cover91_infoTheory}. 
To state it, recall that a probability kernel between two measurable spaces $(\R_1,\mathcal M_1)$ and $(\R_2,\mathcal M_2)$ is
a function $k: \R_1 \times \mathcal M_2 \to [0,1]$ such that $k(\cdot,S)$ is measurable for each $S \in \mathcal M_2$ and
$k(r,\cdot)$ is a probability measure for each $r \in \R_1$.

\begin{thm}[Resilience to post-processing]	\label{thm: resilience to post-processing}
Let $M_1: \D \times \Omega \to (\R_1,\mathcal M_1)$ be an $(\epsilon,\delta)$-differentially private mechanism.
Let $M_2: \D \times \Omega \to (\R_2,\mathcal M_2)$ be another mechanism, such that there exists
a probability kernel $k: \R_1 \times \mathcal M_2 \to [0,1]$ verifying
\begin{align}	\label{eq: post-processing definition}
\Prob(M_2(d) \in S | M_1(d)) = k(M_1(d),S), \; \text{a.s.}, \forall S \in \mathcal M_2, \forall d \in \D.
\end{align}
Then $M_2$ is $(\epsilon,\delta)$-differentially private. 
\end{thm}

Note that in (\ref{eq: post-processing definition}), the kernel $k$ is not allowed to depend on the dataset $d$. In other words, this condition 
says that once $M_1(d)$ is known, the distribution of $M_2(d)$ does not further depend on $d$. The theorem shows that a mechanism $M_2$ 
accessing a dataset only indirectly via the output of a differentially private mechanism $M_1$ cannot weaken the privacy guarantee. 
Hence post-processing can be used freely to improve the \emph{accuracy} of an output, as in Section \ref{section: binary streams} for example, 
without worrying about a possible loss of privacy.


\begin{proof}
To the best of our knowledge, there is no previous proof of the resilience to post-processing theorem available for the case of randomized post-processing 
and $\delta > 0$.
Let $M_1$ be $(\epsilon,\delta)$-differentially private. We have, for two adjacent elements $d, d' \in \D$ and for any $S \in \mathcal M_2$
\ifthenelse {\boolean{TwoColEq}} 
{
\begin{align*}
\Prob(M_2(d) \in S) &= \Exp[\Prob(M_2 \in S | M_1(d))] = \Exp[k(M_1(d),S)] \\
& \leq e^{\epsilon} \Exp[k(M_1(d'),S)]
+ \int_{\R_1} k(m_1,S) \; d \mu(m_1) \\
& = e^{\epsilon} \Prob(M_2(d') \in S) + \nu(S).
\end{align*}
}
{
\begin{align*}
\Prob(M_2(d) \in S) &= \Exp[\Prob(M_2 \in S | M_1(d))] = \Exp[k(M_1(d),S)] \\
& \leq e^{\epsilon} \Exp[k(M_1(d'),S)]
+ \int_{\R_1} k(m_1,S) \; d \mu^{d,d'}(m_1) \\
& = e^{\epsilon} \Prob(M_2(d') \in S) + \nu^{d,d'}(S).
\end{align*}
}
The first equality is just the smoothing property of conditional expectations, and the inequality comes
from (\ref{eq: DP def 2}) applied to the function $k(\cdot,S)$.
Since $k$ is a probability kernel, the integral on the second line defines a measure $\nu^{d,d'}$ on $\R_2$, 
which is $\delta$-bounded since $\nu^{d,d'}(\R_2) = \int_{\R_1} k(m_1,\R_2) d \mu^{d,d'}(m_1) 
= \int_{\R_1} 1 \; d \mu^{d,d'}(m_1) = \mu^{d,d'}(\R_1) \leq \delta$.
\end{proof}


\subsection{Basic Differentially Private Mechanisms}	\label{section: basic mech}

A mechanism that throws away all the information in a dataset is obviously private, 
but not useful, and in general one has to trade off privacy for utility when 
answering specific queries. We recall below two basic mechanisms that can be used to
answer queries in a differentially private way.
We are only concerned in this section with queries that return numerical answers, 
i.e., here a query is a map $q: \D \to \R$, where the output space $\R$ equals $\mathbb R^k$ for some $1 \leq k < \infty$,
is equipped with a norm  denoted $\| \cdot \|_\R$,
and the $\sigma$-algebra $\mathcal M$ on $\R$ is taken to be the standard 
Borel $\sigma$-algebra, denoted $\mathcal R^k$. 
The following quantity plays an important
role in the design of differentially private mechanisms \cite{Dwork06_DPcalibration}.

\begin{definition}	\label{defn: sensitivity}
Let $\D$ be a space equipped with an adjacency relation $\Adj$.
The sensitivity of a query $q: \D \to \R$ is defined as
$
\Delta_\R q := \max_{d,d':\Adj(d,d')} \|q(d) - q(d')\|_\R.
$
In particular, for $\R = \mathbb R^k$ equipped with the $p$-norm 
$\| x \|_p = \left( \sum_{i=1}^k |x_i|^p \right)^{1/p}$ for $p \in [1,\infty]$, we denote the $\ell_p$ sensitivity by $\Delta_p q$.
\end{definition}

\subsubsection{The Laplace Mechanism} This mechanism, proposed in  \cite{Dwork06_DPcalibration}, 
modifies an answer to a numerical query by adding i.i.d. zero-mean noise distributed according 
to a Laplace distribution. Recall that the Laplace distribution with mean zero and scale parameter $b$, 
denoted $\Lap(b)$, has density $p(x;b) = \frac{1}{2b} \exp \left(-\frac{|x|}{b}\right)$ and variance $2 b^2$.
Moreover, for $w \in \mathbb R^k$ with $w_i$ iid and $w_i \sim \Lap(b)$, 
denoted $w \sim \Lap(b)^k$, we have $p(w;b) = (\frac{1}{2b})^k \exp \left(-\frac{\|w\|_1}{b}\right)$, 
 $\Exp[\|w\|_1] = b$, and $\Prob(\|w\|_1 \geq t b) = e^{-t}$. 

\begin{thm}	\label{thm: Laplace mech}
Let $q: \D \to \mathbb R^k$ be a query.
Then the Laplace mechanism $M_q: \D \times \Omega \to \mathbb R^k$ 
defined by $M_q(d) = q(d) + w$, with $w \sim \Lap \left( b \right)^k$ 
and $b \geq \frac{\Delta_1 q}{\epsilon}$ is $\epsilon$-differentially private.
\end{thm}

Note that the mechanism requires \emph{each} coordinate of $w$ to have 
standard deviation proportional to $\Delta_1 q$, as well as inversely proportional to the
privacy parameter $\epsilon$ (here $\delta = 0)$. For example, if $q$ simply consists of $k$ 
repetitions of the same scalar query, then $\Delta_1 q$ increases linearly with $k$, and the 
quadratically growing variance of the noise added to each coordinate prevents an adversary from 
averaging out the noise.

\begin{proof} We have, for $S \subset \mathbb R^k$ measurable 
and $d, d'$ two adjacent datasets in $\D$,
\ifthenelse {\boolean{TwoColEq}} 
{
\begin{align*}
P(M_q(d) \in S) 
&= \left(\frac{1}{2b}\right)^k \int_{\mathbb R^k} 1_S(q(d)+w) e^{-\frac{\|w\|_1}{b}} dw \\
&= \left(\frac{1}{2b}\right)^k \int_{\mathbb R^k} 1_S(u) e^{-\frac{\|u-q(d)\|_1}{b}} dw \\
&\leq e^{\frac{\|q(d)-q(d')\|_1}{b}} \left(\frac{1}{2b}\right)^k \int_{\mathbb R^k} 1_S(u) e^{-\frac{\|u-q(d')\|_1}{b}} dw,
\end{align*}
}
{
\begin{align*}
P(M_q(d) \in S) 
&= \left(\frac{1}{2b}\right)^k \int_{\mathbb R^k} 1_S(q(d)+w) e^{-\frac{\|w\|_1}{b}} dw
= \left(\frac{1}{2b}\right)^k \int_{\mathbb R^k} 1_S(u) e^{-\frac{\|u-q(d)\|_1}{b}} dw \\
&\leq e^{\frac{\|q(d)-q(d')\|_1}{b}} \left(\frac{1}{2b}\right)^k \int_{\mathbb R^k} 1_S(u) e^{-\frac{\|u-q(d')\|_1}{b}} dw,
\end{align*}
}
since $-\|u-q(d)\|_1 \leq -\|u-q(d')\|_1+\|q(d)-q(d')\|_1$ by the triangle inequality.
With the choice of $b = \Delta_1 q/\epsilon$, we obtain the definition (\ref{eq: standard def approximate DP original})
of differential privacy (i.e., with $\delta = 0$).
\end{proof}

\subsubsection{The Gaussian Mechanism}
This mechanism, proposed in \cite{Dwork06_DPgaussian}, is similar to the 
Laplace mechanism but adds i.i.d. Gaussian noise to obtain $(\epsilon,\delta)$-differential privacy, 
with $\delta > 0$ but typically a smaller $\epsilon$ for the same utility. Recall the definition of the 
$\mathcal Q$-function 
\[
\mathcal Q(x) := \frac{1}{\sqrt{2 \pi}} \int_x^{\infty} e^{-\frac{u^2}{2}} du.
\]
The following theorem tightens the analysis from \cite{Dwork06_DPgaussian}.

\begin{thm}	\label{thm: Gaussian mech}
Let $q: \D \to \mathbb R^k$ be a query.
Then the Gaussian mechanism $M_q: \D \times \Omega \to \mathbb R^k$ 
defined by $M_q(d) = q(d) + w$, with $w \sim \mathcal N\left(0,\sigma^2 I_k \right)$, 
where $\sigma \geq \frac{\Delta_2 q}{2 \epsilon}(K + \sqrt{K^2+2\epsilon})$ and $K = \mathcal Q^{-1}(\delta)$,
is $(\epsilon,\delta)$-differentially private.
\end{thm}

\begin{proof}
Let $d, d'$ be two adjacent elements in $\D$, and denote $v := q(d)-q(d')$.
We use the notation $\|\cdot\|$ for the $2$-norm in this proof.
For $S \in \mathcal R^k$, we have
\ifthenelse {\boolean{TwoColEq}} 
{
\begin{align*}
&\Prob(M_q(d) \in S) \\
&= \frac{1}{(2 \pi \sigma^2)^{k/2}} \int_{\mathbb R^k} 1_S(q(d)+w) e^{-\frac{\|w\|^2}{2 \sigma^2}} dw \\
&= \frac{1}{(2 \pi \sigma^2)^{k/2}} \int_{\mathbb R^k} 1_S(u) e^{-\frac{\|u-q(d)\|^2}{2 \sigma^2}} du \\
&= \frac{1}{(2 \pi \sigma^2)^{k/2}} \int_S e^{-\frac{\|u-q(d')\|^2}{2 \sigma^2}} 
e^{\frac{2(u-q(d'))^T v-\|v\|^2}{2 \sigma^2}} du \\
&\leq e^{\epsilon} \Prob(M_q(d') \in S) 
+ 
\frac{1}{(2 \pi \sigma^2)^{k/2}} \int_S \bigg [ e^{-\frac{\|u-q(d)\|^2}{2 \sigma^2}} \\
& \quad \quad 1\{ 2(u-q(d'))^T v  \geq \|v\|^2 + 2 \epsilon \sigma^2 \} \bigg ] du.
\end{align*}
}
{
\begin{align*}
&\Prob(M_q(d) \in S) = \frac{1}{(2 \pi \sigma^2)^{k/2}} \int_{\mathbb R^k} 1_S(q(d)+w) e^{-\frac{\|w\|^2}{2 \sigma^2}} dw 
= \frac{1}{(2 \pi \sigma^2)^{k/2}} \int_{\mathbb R^k} 1_S(u) e^{-\frac{\|u-q(d)\|^2}{2 \sigma^2}} du \\
&= \frac{1}{(2 \pi \sigma^2)^{k/2}} \int_S e^{-\frac{\|u-q(d')\|^2}{2 \sigma^2}} 
e^{\frac{2(u-q(d'))^T v-\|v\|^2}{2 \sigma^2}} du \\
&\leq e^{\epsilon} \Prob(M_q(d') \in S) + 
\frac{1}{(2 \pi \sigma^2)^{k/2}} \int_S \bigg [ e^{-\frac{\|u-q(d)\|^2}{2 \sigma^2}} 
1\Big \{ 2(u-q(d'))^T v  \geq \|v\|^2 + 2 \epsilon \sigma^2 \Big \} \bigg ] du.
\end{align*}
}
The last integral term defines a measure $S \mapsto \mu^{d,d'}(S)$ on $\mathbb R^k$
that we wish to bound by $\delta$.
%
With the change of variables $y=(u-q(d))/\sigma$ and the choice $S = \mathbb R^k$ in the integral, 
we can rewrite it as
$
\Prob(Y^T v \geq \epsilon \sigma - \|v\|^2/2\sigma ),
$
with $Y \sim \mathcal N(0,I_k)$. In particular, $Y^T v \sim \mathcal N(0,\|v\|^2)$, 
hence is equal to $\|v\| Z$ in distribution, with $Z \sim \mathcal N(0,1)$.
We are then led to set $\sigma$ sufficiently large so that
$\Prob(Z \geq \epsilon \sigma/\|v\| - \|v\|/2\sigma) \leq \delta$, i.e., 
$\mathcal Q(\epsilon \sigma/\|v\| - \|v\|/2\sigma) \leq \delta$.
The result then follows by straightforward calculation.
\end{proof}

As an illustration of the theorem, to guarantee $(\epsilon,\delta)$-differential privacy with $\epsilon = \ln 2$ and 
$\delta = 0.05$, the standard deviation of the Gaussian noise should be about $2.65$ times the $\ell_2$ sensitivity 
of $q$. For the rest of the paper, we define 
$
\kappa(\delta,\epsilon) = \frac{1}{2 \epsilon} (K+\sqrt{K^2+2\epsilon}),
$
so that the standard deviation $\sigma$ in Theorem \ref{thm: Gaussian mech}
can be written $\sigma(\delta,\epsilon) = \kappa(\epsilon,\delta) \Delta_2 q$.
It can be shown that $\kappa(\delta,\epsilon)$ can be bounded by $O(\ln(1/\delta))^{1/2}/\epsilon$.


\section{Differentially Private Dynamic Systems}		\label{section: DP linear systems}

In this section we introduce the notion of differential privacy for dynamic systems. We start with some 
notations and technical prerequisites. All signals are discrete-time signals, start at time $0$, and all 
systems are assumed to be causal. For each time $T$, let $P_T$ be the truncation operator, so that 
for any signal $x$ we have
\[
(P_T x)_t = \begin{cases}
x_t, & t \leq T \\
0, & t > T.
\end{cases}
\]
Hence a deterministic system $\mathcal G$ is causal if and only if $P_T \mathcal G = P_T \mathcal G P_T$.
We denote by $\ell_{p,e}^m$ the space of sequences with values in
$\mathbb R^m$ and such that $x \in \ell_{p,e}^m$ if and only if $P_T x$
has finite $p$-norm for all integers $T$. The $\mathcal H_2$ norm and $\mathcal H_\infty$
norm of a stable transfer function $\mathcal G$ are defined respectively as
$
\|\mathcal G\|_2 = \left(\frac{1}{2 \pi} \int_{-\pi}^\pi \Tr (\mathcal G^*(e^{i \omega}) \mathcal G(e^{i \omega})) d \omega\right)^{1/2}, 
\|\mathcal G\|_\infty = \text{ess} \sup_{\omega \in [-\pi,\pi)} \sigma_{\max} (\mathcal G(e^{i \omega})),
$
where $\sigma_{\max}(A)$ denotes the maximum singular value of a matrix $A$.

We consider situations in which private participants contribute input signals 
driving a dynamic system and the queries consist of output signals 
of this system. 
First, in this section, we assume that the input of a system consists of $n$ signals, one
for each participant. An input signal is denoted
$u=(u_1,\ldots,u_n)$, with $u_{i} \in \ell_{r_i,e}^{m_i}$ for some $m_i \in \mathbb N$ and
$r_i \in [1,\infty]$. 
A simple example is that of a dynamic system releasing at each period
the average over the past $l$ periods of the sum of the input values 
of the participants, i.e., with output $\frac{1}{l} \sum_{k=t-l+1}^t \sum_{i=1}^n u_{i,k}$ at time $t$,
see Fig. \ref{fig: MA example}.
For $r=(r_1,\ldots,r_n)$ and $m=(m_1,\ldots,m_n)$, an adjacency relation can be defined
on $l_{r,e}^m = \ell_{r_1,e}^{m_1} \times \ldots \times \ell_{r_n,e}^{m_n}$ for example by
$\Adj(u,u')$ if and only if $u$ and $u'$ differ by exactly one component signal, and moreover
this deviation is bounded. That is, let us fix a set of nonnegative numbers 
$b=(b_1, \ldots, b_n)$, $b_i \geq 0$, and define 
\ifthenelse {\boolean{TwoColEq}} 
{
\begin{align}	\label{def: adjacency for vector signals}
\Adj^b(u,u') \text{ iff for some } i, \|u_i - u_i'\|_{r_i} \leq b_i, \\
\text{and } u_j = u_j' \text{ for all } j \neq i. \nonumber
\end{align}
}
{
\begin{align}	\label{def: adjacency for vector signals}
\Adj^b(u,u') \text{ iff for some } i, \|u_i - u_i'\|_{r_i} \leq b_i, 
\text{ and } u_j = u_j' \text{ for all } j \neq i. 
\end{align}
}

\begin{figure}
\centering
\includegraphics[height=3cm]{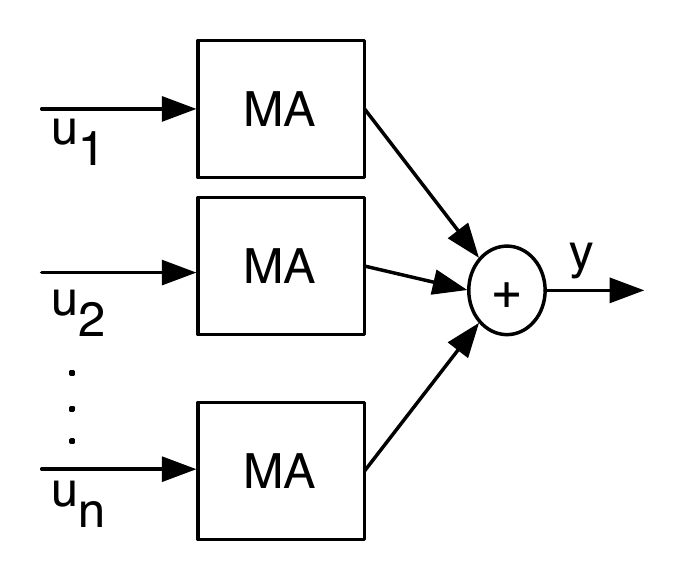}
\caption{Illustrative example of a system computing the sum of the moving averages (MA) of input signals
contributed by $n$ individual participants. A differentially private version of this system, for the adjacency relation
(\ref{def: adjacency for vector signals}), will guarantee to user $i$ that the distribution of the output signal does not
vary significantly when her input varies in $r_i$-norm by at most $b_i$. In particular, the distribution of the output signal
will not change significantly if user $i$'s input is zero ($u_i \equiv 0$, e.g., because the user is not present), or is 
not zero but satisfies $\|u_i\|_{r_i} \leq b_i$.
}
\label{fig: MA example}
\end{figure}

\subsection{Finite-Time Criterion for Differential Privacy}

To approximate dynamic systems by versions respecting the differential privacy of the individual participants, 
we consider mechanisms of the form $M: \ell_{r,e}^{m} \times \Omega \to  \ell_{s,e}^{m'}$, i.e., producing
for any input signal $u \in \ell_{r,e}^{m}$ a stochastic process $Mu$ with sample paths in $\ell_{s,e}^{m'}$. 
As in the previous section, this requires that we first specify the measurable sets of $\ell_{s,e}^{m'}$. 
We start by defining in a standard way the measurable sets of $(\mathbb R^{m'})^\mathbb N$,
the space of sequences with values in $\mathbb R^{m'}$, to be the $\sigma$-algebra
denoted $\mathcal M^{m'}$ generated by the so-called finite-dimensional cylinder sets of the form
$
\{y \in (\mathbb R^{m'})^\mathbb N: y_{0:T} \in H_T\}, \text{for } T \geq 0$ and $H_T \in \mathcal R^{(T+1) m'},
$
where $y_{0:T}$ denotes the vector $[y_0^T,\ldots,y_T^T]^T$ (see, e.g., \cite[chapter 2]{Breiman92_book}).
The measurable sets considered for the output of $M$ are then obtained by intersection of 
$\ell_{s,e}^{m'}$ with the sets of $\mathcal M^{m'}$. The resulting $\sigma$-algebra is denoted 
$\mathcal M_{s,e}^{m'}$ and is generated by the sets of the form
\begin{align}	\label{eq: finite-dimensional meas. sets}
\tilde H_T = \{y \in \ell_{s,e}^{m'}: y_{0:T} \in H_T\}, \text{ for } T \geq 0, H_T \in \mathcal R^{(T+1) m'}.
\end{align}

As for the dynamic systems of interest, we constrain in this paper the mechanisms to be 
causal, i.e., the distribution of $P_T M u$ should be the same as that of $P_T M P_T u$
for any $u \in \ell_{r,e}^{m}$ and any time $T$. In other words, the values $u_t$ for $t > T$ 
do not influence the values of the mechanism output up to time $T$.
The following technical lemma is useful to show that a mechanism on signal spaces 
is $(\epsilon,\delta)$-differentially private by considering only finite dimensional problems. 

\begin{lem}	\label{lem: technical FTDP condition}
Consider an adjacency relation $\text{Adj}$ on $\ell_{r,e}^{m}$.
For a mechanism $M: \ell_{r,e}^{m} \times \Omega \to  \ell_{s,e}^{m'}$, the following are equivalent 
\begin{enumerate}[(a)]	
\item $M$ is $(\epsilon,\delta)$-differentially private.		\label{FTDP - enum 1}
\item \label{FTDP - enum 2}
For all $u, u'$ in $\ell_{r,e}^{m}$ such that $\text{Adj}(u,u')$, we have
\begin{align}	\label{eq: finite-time DP}
\Prob((Mu)_{0:T} \in A) \leq e^{\epsilon} \, \Prob((Mu')_{0:T} \in A) + \delta, 
\;\; \forall T \geq 0, \forall A \in \mathcal R^{(T+1) m'}.
\end{align}
\end{enumerate}
\end{lem}


\begin{proof}
$\ref{FTDP - enum 1}) \Rightarrow \ref{FTDP - enum 2})$ 
If $M$ is $(\epsilon,\delta)$-differentially private, then for $u,u'$ adjacent, and for all $H \in \mathcal M_{s,e}^{m'}$, 
we have $\Prob(Mu \in H) \leq e^{\epsilon} \; \Prob(Mu' \in H) + \delta$. In particular, for a given integer $T \geq 0$, 
we can restrict our attention to the sets $\tilde H_T$ of the form (\ref{eq: finite-dimensional meas. sets}).
In this case, we have immediately $\Prob(Mu \in \tilde H_T) = \Prob((Mu)_{0:T} \in H_T)$ since the events
are the same.

$\ref{FTDP - enum 2}) \Rightarrow \ref{FTDP - enum 1})$ 
Conversely, consider two adjacent signal $u,u' \in \ell_{r,e}^{m}$, and let $S \in \mathcal M^{m'}_{s,e}$, for 
which we want to show (\ref{eq: standard def approximate DP original}). Fix $\eta > 0$. There exists $T \geq 0$ 
and $H_T \in \mathcal R^{(T+1) m'}$ such that $\Prob(Mu \in S \Delta \tilde H_T) \leq  \eta$ 
and $\Prob(Mu' \in S \Delta \tilde H_T) \leq \eta$, where $A \Delta B := (A \setminus B) \cup (B \setminus A)$ denotes the 
symmetric difference. This is a consequence for example of the fact that the finite-dimensional cylinder sets form an algebra
and of the argument in the proof of \cite[Theorem 3.1.10]{Dudley02_book}. We then have
\ifthenelse {\boolean{TwoColEq}} 
{
\begin{align*}
\mathbb P(Mu \in A) &\leq \mathbb P(Mu \in \tilde H_T) + \eta \\
&= \mathbb P((Mu)_{0:T} \in H_T) + \eta \\
&\leq e^{\epsilon} \, \mathbb P((Mu')_{0:T} \in H_T) + \mu^{u,u'}_T(H_T) + \eta \\
&=e^{\epsilon} \, \mathbb P(Mu' \in \tilde H_T) + \mu^{u,u'}(\tilde H_T) + \eta \\
&\leq e^{\epsilon} \, \mathbb P(Mu' \in A) + \mu(A) + \eta (2 + e^\epsilon). 
\end{align*}
}
{
\begin{align*}
\mathbb P(Mu \in S) &\leq \mathbb P(Mu \in \tilde H_T) + \eta 
= \mathbb P((Mu)_{0:T} \in H_T) + \eta \\
&\leq e^{\epsilon} \, \mathbb P((Mu')_{0:T} \in H_T) + \delta + \eta 
=e^{\epsilon} \, \mathbb P(Mu' \in \tilde H_T) + \delta + \eta \\
&\leq e^{\epsilon} \, \mathbb P(Mu' \in S) + \delta + \eta (1 + e^\epsilon). 
\end{align*}
}
Since $\eta$ can be taken arbitrarily small, the differential privacy definition
(\ref{eq: standard def approximate DP original}) holds.
\end{proof}

\subsection{Basic Dynamic Mechanisms}

Recall (see, e.g., \cite{VanderSchaft00_passivity}) that for a system $\mathcal G$ with inputs in $\ell_{r,e}^{m}$ 
and output in $\ell_{s,e}^{m'}$, its $\ell_{r}$-to-$\ell_{s}$ incremental gain $\gamma^{inc}_{r,s}(\mathcal G)$ is defined 
as the smallest number $\gamma$ such that
\[
\| P_T \mathcal G u- P_T \mathcal G u' \|_s \leq \gamma \| P_T u - P_T u' \|_r, \;\; \forall u, u' \in \ell_{r,e}^m, \; \forall T.
\]
Now consider, for $r=(r_1,\ldots,r_n)$ and $m=(m_1,\ldots,m_n)$, a system $\mathcal G: l^m_{r,e} \to \ell_{s,e}^{m'}$ defined by
\begin{align}
\mathcal G (u_1,\ldots,u_n) = \sum_{i=1}^n \mathcal G_i u_i,	\label{eq: general system considered}
\end{align}
where $\mathcal G_i: \ell_{r_i,e}^{m_i} \to \ell_{s,e}^{m'}$, for all $1 \leq i \leq n$. The next theorem generalizes the Laplace and
Gaussian mechanisms of Theorems \ref{thm: Laplace mech} and \ref{thm: Gaussian mech} to causal dynamic systems.

\begin{thm}	\label{thm: differentially private large scale system}
Let $\mathcal G$ be defined as in (\ref{eq: general system considered}) and
consider the adjacency relation (\ref{def: adjacency for vector signals}).
Then the mechanism $Mu = \mathcal Gu+w$, where $w$ is a white noise with $w_t \sim \Lap(B/\epsilon)^{m'}$ and
$B \geq \max_{1 \leq i \leq n} \{ \gamma^{inc}_{r_i,1}(\mathcal G_i) \, b_i \}$, is $\epsilon$-differentially private.
The mechanism is $(\epsilon,\delta)$-differentially private
if $w_t \sim \mathcal N(0,\sigma^2 I_{m'})$, 
with $\sigma \geq \kappa(\delta,\epsilon) \max_{1 \leq i \leq n} \{ \gamma^{inc}_{r_i,2}(\mathcal G_i) \, b_i \}$.
\end{thm}

\begin{proof}
Consider two adjacent signals $u, u'$, differing say in their $i^{\text{th}}$ component.
Then, for $\alpha \in \{1,2\}$, we have
\ifthenelse {\boolean{TwoColEq}} 
{
\begin{align*}
\|P_T \mathcal G u - P_T \mathcal G u'\|_\alpha &= \|P_T \mathcal G_i u_i - P_T \mathcal G_i u_i'\|_\alpha \\
& \leq \gamma_{r_i,\alpha} \|P_T u_i - P_T u_i'\|_{r_i} \\
& \leq \gamma_{r_i,\alpha} \|u_i - u_i'\|_{r_i} \\
& \leq \gamma_{r_i,\alpha} b_i.
\end{align*}
}
{
\begin{align*}
\|P_T \mathcal G u - P_T \mathcal G u'\|_\alpha &= \|P_T \mathcal G_i u_i - P_T \mathcal G_i u_i'\|_\alpha 
\leq \gamma_{r_i,\alpha} \|P_T u_i - P_T u_i'\|_{r_i} \\
& \leq \gamma_{r_i,\alpha} \|u_i - u_i'\|_{r_i} 
\leq \gamma_{r_i,\alpha} b_i.
\end{align*}
}
This leads to a bound on the $\ell_1$ and $\ell_2$ sensitivity of $P_T \mathcal G$, valid for all $T$.
The result is then an application of Theorems \ref{thm: Laplace mech} and \ref{thm: Gaussian mech}
and Lemma \ref{lem: technical FTDP condition}, since (\ref{eq: finite-time DP}) is satisfied for all
$T$.
\end{proof}

 \begin{cor}	\label{cor: differentially private large scale system - linear}
 Let $\mathcal G$ be defined as in (\ref{eq: general system considered}) with each
 system $\mathcal G_i$ linear, and $r_i = 2$ for all $1 \leq i \leq n$.
Then the mechanism $Mu = \mathcal Gu+w$, where $w$ is a white Gaussian noise 
with $w_t \sim \mathcal N(0,\sigma^2 I_{m'})$ and
$\sigma \geq \kappa(\delta,\epsilon) \max_{1 \leq i \leq n} \{ \|\mathcal G_i\|_\infty \, b_i \}$,
is $(\epsilon,\delta)$-differentially private for (\ref{def: adjacency for vector signals}).
\end{cor}

\subsection{Filter Approximation Set-ups for Differential Privacy}		\label{section: approximation set-ups}

\begin{figure}
\centering
\includegraphics[height=2.75cm]{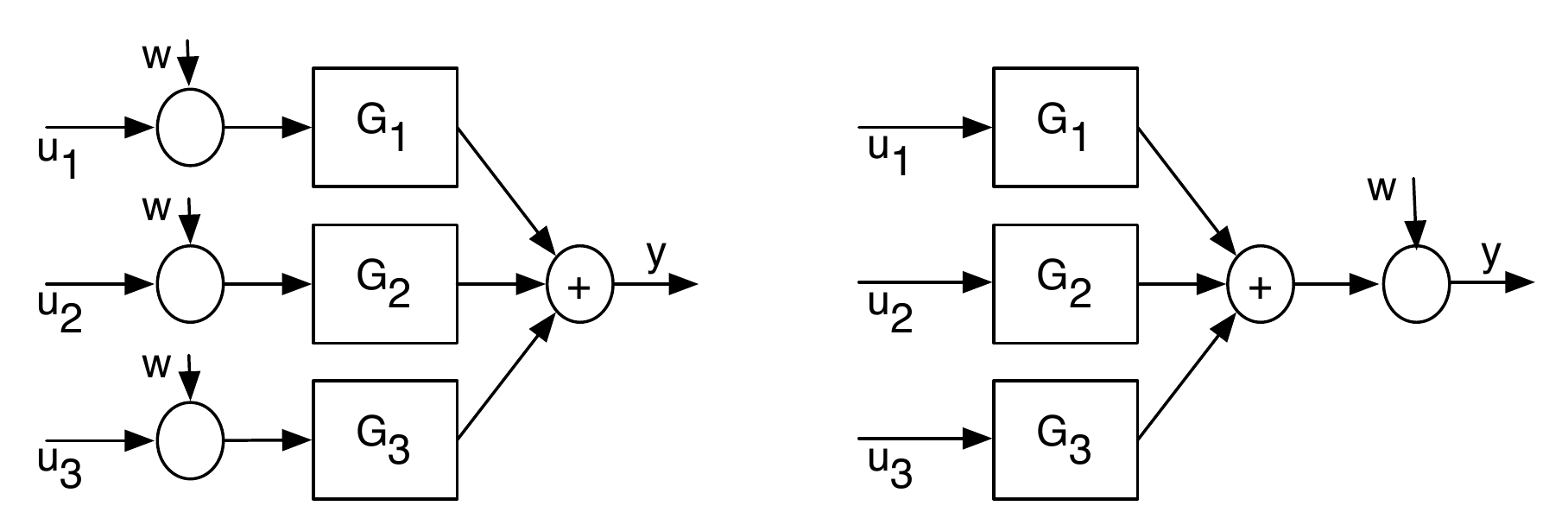}
\caption{Two architectures for differential privacy. (a) Input perturbation. (b) Output perturbation.}
\label{fig: DP basic architectures}
\end{figure}

Let $r_i = 2$ for all $i$ and $\mathcal G$ be linear as in the Corollary \ref{cor: differentially private large scale system - linear}, 
and assume for simplicity the same bound $b_1^2 = \ldots = b_n^2 = B$
for the allowed variations in energy of each input signal. 
We have then two simple mechanisms producing a differentially private version of $\mathcal G$,
depicted on Fig.~\ref{fig: DP basic architectures}. 
The first one directly perturbs each input signal $u_i$ by adding to it a white Gaussian noise $w_i$
with $w_{i,t} \sim \mathcal N(0,\sigma^2 I_{m_i})$ and $\sigma^2 = \kappa(\delta,\epsilon)^2 B$.
These perturbations on each input channel are then passed through $\mathcal G$,
leading to a mean squared error (MSE) for the output equal to 
$\kappa(\delta,\epsilon)^2 B \|\mathcal G\|^2_2 
= \kappa(\delta,\epsilon)^2 B \sum_{i=1}^n \|\mathcal G_i\|^2_2$. 
Alternatively, we can add a single source of noise at the output of $\mathcal G$ 
according to Corollary \ref{cor: differentially private large scale system - linear},
in which case the MSE is 
$\kappa(\delta,\epsilon)^2 B \max_{1 \leq i \leq n} \{ \|\mathcal G_i\|_\infty^2 \}$.
Both of these schemes should be evaluated depending on the system $\mathcal G$
and the number $n$ of participants, as none of the error bound is better 
than the other in all circumstances. For example, if $n$ is small or if
the bandwidths of the individual transfer functions $\mathcal G_i$ do not overlap, 
the error bound for the input perturbation scheme can be smaller. 
Another advantage of this scheme is that the users can release 
differentially private signals themselves without relying on a trusted server.
However, there are cryptographic means for achieving the output perturbation
scheme without centralized trusted server as well, see, e.g., \cite{Shi11_DPaggregation}.

\begin{exmp}
Consider again the problem of releasing the average over the past $l$ periods
of the sum of the input signals, i.e., $\mathcal G = \sum_{i=1}^n \mathcal G_i$ with
$(\mathcal G_i u_i)_t = \frac{1}{l} \sum_{k=t-l+1}^t u_{i,k}$,
for all $i$. Then $\|\mathcal G_i\|_2^2 = 1/l$, whereas $\|\mathcal G_i\|_\infty = 1$, for all $i$.
The MSE for the scheme with the noise at the input is then 
$\kappa(\delta,\epsilon)^2 B n/l$. 
With the noise at the output, the MSE is $\kappa(\delta,\epsilon)^2 B$,
which is better exactly when $n> l$, i.e., the number of users is larger than the
averaging window.
\end{exmp}


\section{Differentially Private Kalman Filtering}		\label{section: private KF}

We now discuss the Kalman filtering problem subject to a differential privacy constraint. 
Compared to the previous section, for Kalman filtering it is assumed that more is 
publicly known about the dynamics of the processes producing the individual signals. 
The goal here is to guarantee differential privacy for the individual state trajectories. 
Section \ref{section: traffic monitoring example} describes an application
of the privacy mechanisms presented here to a traffic monitoring problem.

\subsection{A Differentially Private Kalman Filter}

Consider a set of $n$ linear systems, each with independent dynamics
\begin{align}	\label{eq: linear dynamics participant i}
x_{i,t+1} = A_i x_{i,t} + B_i w_{i,t}, \;\; t \geq 0, \;\; 1 \leq i \leq n,
\end{align}
where $w_i$ is a standard zero-mean Gaussian white noise process with covariance $\Exp[w_{i,t} w_{i,t'}] = \delta_{t-t'}$,
and the initial condition $x_{i,0}$ is a Gaussian random variable with mean $\bar x_{i,0}$, 
independent of the noise process $w_i$. 
System $i$, for $1 \leq i \leq n$, sends measurements 
\begin{align}	\label{eq: measurements participant i}
y_{i,t} = C_i x_{i,t} + D_i w_{i,t}
\end{align}
to a data aggregator. We assume for simplicity that the matrices $D_i$ are full row rank.
%
Figure \ref{fig: KF setup} shows this initial set-up.

The data aggregator aims at releasing a signal that asymptotically 
minimizes the minimum mean squared error with respect to a linear combination of the individual states. 
That is, the quantity of interest to be estimated at each period is $z_t = \sum_{i=1}^n L_i x_{i,t}$, 
where $L_i$ are given matrices, and we are looking for a causal estimator $\hat z$ constructed from the 
signals $y_i, 1 \leq i \leq n$, solution of
\[
\min_{\hat z} \lim_{T \to \infty} \frac{1}{T} \sum_{t=0}^{T-1} \Exp \left[ \|z_t - \hat z_t\|_2^2 \right].
\]
The data $\bar x_{i,0}, A_i, B_i, C_i, D_i, L_i, 1 \leq i \leq n,$ are assumed to be public information.
For all $1 \leq i \leq n$, we assume that the pairs $(A_i, C_i)$ are detectable and the pairs
$(A_i, B_i)$ are stabilizable. 
In the absence of privacy constraint, the optimal estimator is $\hat z_t = \sum_{i=1}^n L_i \hat x_{i,t}$, 
with $\hat x_{i,t}$ provided by the steady-state Kalman filter estimating the state of system $i$ from $y_i$  \cite{Anderson05_filtering},
and denoted $\mathcal K_i$ in the following. 

\begin{figure}
\centering
\includegraphics[height=3.5cm]{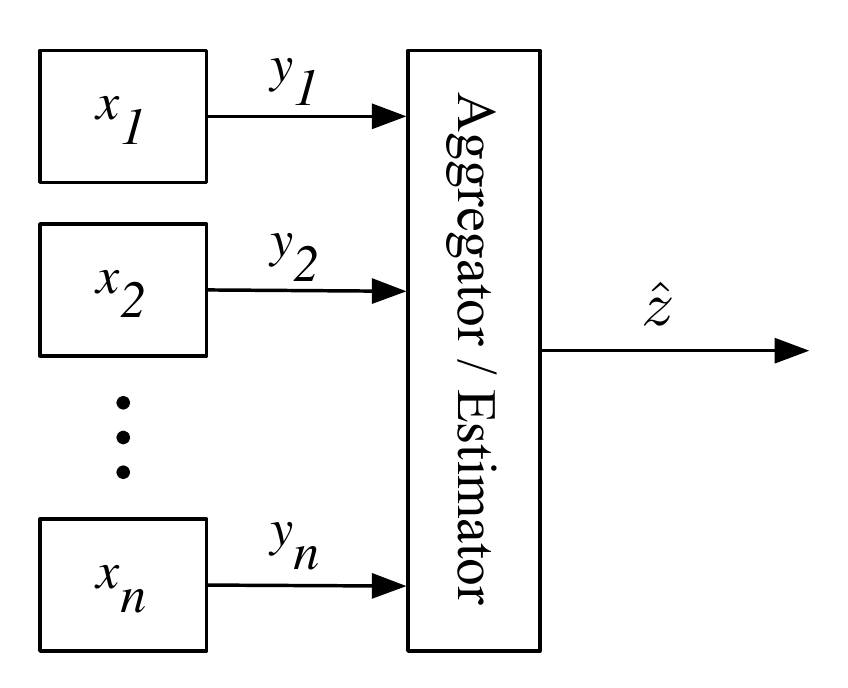}
\caption{Kalman filtering set-up.}
\label{fig: KF setup}
\end{figure}

Suppose now that the publicly released estimate should guarantee the differential privacy
of the participants. This requires that we first specify an adjacency relation on the appropriate space
of datasets. Let $x = [x_1^T, \ldots, x_n^T]^T$ and $y = [y_1^T, \ldots, y_n^T]^T$ denote
the global state and measurement signals. 
Assume that the mechanism is required to guarantee differential privacy with respect
to a subset $\mathcal S_i := \{i_1, \ldots, i_k\}$ of the coordinates of the state trajectory $x_i$.
Let the selection matrix $S_i$ be the diagonal matrix with $[S_i]_{jj} = 1$ if $j \in \mathcal S_i$, and $[S_i]_{jj} = 0$ otherwise. 
Hence $S_i v$ sets the coordinates of a vector $v$ which do not belong to the set $\mathcal S_i$ to zero.
Fix a vector $\rho \in \mathbb R_+^n$. The adjacency relation considered here is
\ifthenelse {\boolean{TwoColEq}} 
{
\begin{align} \label{eq: adjacency for state trajectories - bis}
&\Adj_{\mathcal S}^\rho(x,x') \text{ iff } \text{for some } i, \; \|S_ix_i - S_ix_i' \|_2 \leq \rho_i, \\
& (I-S_i)x_i = (I-S_i)x_i', \text{and } x_j = x_j' \text{ for all } j \neq i. \nonumber
\end{align}
}
{
\begin{align} \label{eq: adjacency for state trajectories - bis}
\Adj_{\mathcal S}^\rho(x,x') \text{ iff } \text{for some } i, &\; \|S_ix_i - S_ix_i' \|_2 \leq \rho_i, 
(I-S_i)x_i = (I-S_i)x_i', \\
& \text{and } x_j = x_j' \text{ for all } j \neq i. \nonumber
\end{align}
}
In words, two adjacent global state trajectories differ by the values of a single participant, say $i$.
Moreover, for differential privacy guarantees we are constraining the range in energy variation
in the signal $S_i x_i$ of participant $i$ to be at most $\rho_i^2$. Hence, the distribution 
on the released results should be essentially the same if a participant's state signal value 
$S_i x_{i,t_0}$ at some single specific time $t_0$ were replaced by $S_i x'_{i,t_0}$ with 
$\|S_i (x_{i,t_0} - x'_{i,t_0})\| \leq \rho_i$, but the privacy guarantee should also hold for 
smaller instantaneous deviations on longer segments of trajectory. 
Other adjacency relations could be considered, e.g., directly on the measured signals $y$ or more generally 
on linear combinations of the components of individual states.

Depending on which signals on Fig. \ref{fig: KF setup} are actually published, and similarly to the discussion
of Section \ref{section: approximation set-ups}, there are different points at which a privacy inducing noise can be introduced.
First, for the input noise injection mechanism, the noise can be added by each participant directly 
to their transmitted measurement signal $y_i$.
Namely, since for two state trajectories $x_i, x_i'$ adjacent according to (\ref{eq: adjacency for state trajectories - bis}) we have
$x_i - x_i' = S_i(x_i-x_i')$, the variation for the corresponding measured signals can be bounded as follows
\[
\|y_i - y_i'\|_2 = \|C_i S_i (x_i - x_i') \|_2 = \|C_i S_i S_i (x_i - x_i') \|_2 \leq \sigma_{\max}(C_i S_i) \rho_i.
\]
Hence differential privacy can be guaranteed if participant $i$ adds to $y_i$ a white Gaussian noise with covariance matrix 
$\kappa(\delta,\epsilon)^2 \rho_i^2 \sigma^2_{\max} (C_i S_i) I_{p_i}$, where $p_i$ is the dimension of $y_{i,t}$.
Note that in this sensitivity computation the measurement noise $D_i w_i$ has the same realization independently 
of the considered variation in $x_i$.
At the data aggregator, the privacy-preserving noise can be taken into account in the design of the Kalman filter, 
since it can be viewed as an additional measurement noise.
Again, an advantage of this mechanism is its simplicity of implementation when the participants 
do not trust the data aggregator, since the transmitted signals are already differentially private.

Next, consider the output noise injection mechanism. 
Since we assume that $\bar x^i_0$ is public information, the initial condition $\hat x_{i,0}$ of each state estimator is fixed. 
Consider now two state trajectories $x, x'$, adjacent according to (\ref{eq: adjacency for state trajectories - bis}), 
and let $\hat z, \hat z'$ be the corresponding estimates produced by the Kalman filters. We have
\[
\hat z - \hat z' = L_i \mathcal K_i (y_i - y_i') = L_i \mathcal K_i C_i S_i (x_i - x_i'),
\]
where we recall that $\mathcal K_i$ is the $i^{th}$ Kalman filter. Hence
$
\| \hat z - \hat z' \|_2 \leq \gamma_i \rho_i,
$
where $\gamma_i$ is the $\mathcal H_\infty$ norm of the transfer function $L_i \mathcal K_i C_i S_i$.
We thus have the following theorem.

\begin{thm}		\label{thm: general output perturbation mechanism result}
A mechanism releasing $\left( \sum_{i=1}^n L_i \mathcal K_i y_i \right) + \gamma \; \kappa(\delta,\epsilon) \; \nu$, 
where $\nu$ is a standard white Gaussian noise independent of $\{w_i\}_{1 \leq i \leq n}, \{x_{i,0}\}_{1 \leq i \leq n}$, 
and $\gamma = \max_{1 \leq i \leq n} \{\gamma_i \rho_i\}$, with $\gamma_i$ the $\mathcal H_\infty$ norm of 
$L_i \mathcal K_i C_i S_i$, is differentially private for the adjacency relation (\ref{eq: adjacency for state trajectories - bis}).
\end{thm}


\subsection{Filter Redesign for Stable Systems}	\label{section: filter redesign - stable}

In the case of the output perturbation mechanism, one can potentially improve the MSE performance
of the filter with respect to the Kalman filter used in the previous subsection.
Namely, consider the design of $n$ filters of the form
\begin{align}
\hat x_{i,t+1} &= F_i \hat x_{i,t} + G_i y_{i,t}  	\label{eq: filter state} \\
\hat z_{i,t} &= H_i \hat x_{i,t} + K_i y_{i,t},	\label{eq: filter output}
\end{align}
for $1 \leq i \leq n$, where $F_i, G_i, H_i, K_i$ are matrices to determine. The estimator considered is
$\hat z_t = \sum_{i=1}^n \hat z_{i,t}$,
so that each filter output $\hat z_{i}$ should minimize the steady-state MSE with $z_{i} = L_i x_{i}$,
and the released signal should guarantee differential privacy with respect to 
(\ref{eq: adjacency for state trajectories - bis}). Assume first in this section that the system matrices $A_i$ 
are stable, in which case we also restrict the filter matrices $F_i$ to be stable.
Moreover, we only consider the design of full order filters, i.e., the dimensions of $F_i$ are greater or equal to those
of $A_i$, for all $1 \leq i \leq n$. 

Denote the overall state for each system and associated filter by $\tilde x_i = [x_i^T, \hat x_i^T]^T$.
The combined dynamics from $w_i$ to the estimation error $e_i := z_i - \hat z_i$ can be written
\begin{align*}
\tilde x_{i,t+1} &=  \tilde A_i \tilde x_{i,t} + \tilde B_i w_{i,t} \\
e_{i,t} &= \tilde C_i \tilde x_{i,t} + \tilde D_i w_{i,t},
\end{align*}
where
\[
\tilde A_i = \begin{bmatrix}
A_i & 0 \\
G_iC_i & F_i
\end{bmatrix}, \;\;
\tilde B_i = \begin{bmatrix}
B_i \\ G_i D_i
\end{bmatrix}, \;\;
\tilde C_i = \begin{bmatrix}
L_i - K_i C_i & -H_i
\end{bmatrix}, \;\;
\tilde D_i = -K_i D_i.
\]
The steady-state MSE for the $i^{th}$ estimator is then $\lim_{t \to \infty} \mathbb E[e_{i,t}^T e_{i,t}]$.
Moreover, we are interested in designing filters with small $\mathcal H_\infty$ norm, in order to minimize the
amount of noise introduced by the privacy-preserving mechanism, which ultimately also impacts the
overall MSE. Considering as in the previous subsection the sensitivity of filter $i$'s output to a change
from a state trajectory $x$ to an adjacent one $x'$ according to (\ref{eq: adjacency for state trajectories - bis}),
and letting $\delta x_i = x_i - x_i' = S_i (x_i - x_i') = S_i \delta x_i$, we see that the change in the output of filter $i$ 
follows the dynamics
\begin{align*}
\delta \hat x_{i,t+1} &= F_i \delta \hat x_{i,t} + G_i C_i S_i \delta x_i \\
\delta \hat z_i &= H_i \delta \hat x_{i,t} + K_i C_i S_i \delta x_i.
\end{align*}
Hence the $\ell_2$-sensitivity can be measured by the $\mathcal H_\infty$ norm of the transfer
function
\begin{align}	
\TF{F_i}{G_i C_i S_i}{H_i}{K_i C_i S_i}.
\end{align}

Simply replacing the Kalman filter in Theorem \ref{thm: general output perturbation mechanism result}, 
the MSE for the output perturbation mechanism guaranteeing $(\epsilon,\delta)$-privacy is then
\begin{align*}
\sum_{i=1}^n \|\tilde C_i (zI - \tilde A_i)^{-1} \tilde B_i + \tilde D_i \|_2^2 + \kappa(\delta,\epsilon)^2 
\max_{1\leq i \leq n} \{\gamma_i^2 \rho_i^2\}, \\
\text{with } \gamma_i := \| H_i (zI - F_i)^{-1} G_i C_i S_i + K_i C_i S_i \|_\infty.
\end{align*}
Hence minimizing this MSE leads us to the following optimization problem
\begin{align}
&\min_{\mu_i, \lambda, F_i, G_i, H_i, K_i} \quad 
\sum_{i=1}^n \mu_i + \kappa(\delta,\epsilon)^2 \lambda \label{eq: optimization filter} \\
& \text{s.t. } \forall \; 1 \leq i \leq n, \|\tilde C_i (zI - \tilde A_i)^{-1} \tilde B_i + \tilde D_i \|_2^2 \leq \mu_i, 
 \label{eq: optimization filter - constraint H2} \\
& \rho_i^2 \| H_i (zI - F_i)^{-1} G_i C_i S_i + K_i C_i S_i \|^2_\infty \leq \lambda.
\label{eq: optimization filter - constraint Hinf}
\end{align}
Assume without loss of generality that $\rho_i > 0$ for all $i$, since the privacy constraint 
for the signal $x_i$ vanishes if $\rho_i = 0$. 
The following theorem gives a convex sufficient condition in the form of Linear Matrix Inequalities (LMIs) 
guaranteeing that a choice of filter matrices $F_i, G_i, H_i, K_i$ satisfies the constraints 
(\ref{eq: optimization filter - constraint H2})-(\ref{eq: optimization filter - constraint Hinf}).

\begin{thm}	\label{thm: LMI constraints filter design}
The constraints (\ref{eq: optimization filter - constraint H2})-(\ref{eq: optimization filter - constraint Hinf}), for some $1 \leq i \leq n$,
are satisfied if there exists matrices $W_i, Y_i, Z_i, \hat F_i, \hat G_i, \hat H_i, \hat K_i$ such that $\Tr(W_i) < \mu_i$, 
\begin{align*}
\begin{bmatrix}
W_i & (L_i - \hat K_i C_i - \hat H_i) & (L_i - \hat K_i C_i) & -\hat K_i D_i \\
* & Z_i & Z_i & 0 \\
* & * & Y_i & 0 \\
* & * & * & I
\end{bmatrix} \succ 0, \\
\end{align*}
\begin{align*}
\begin{bmatrix}
Z_i & Z_i & Z_i A_i & Z_i A_i & Z_i B_i \\
* & Y_i & (Y_i A_i + \hat G_i C_i + \hat F_i) & (Y_i A_i + \hat G_i C_i) & (Y_i B_i + \hat G_i D_i) \\
* & * & Z_i & Z_i & 0 \\
* & * & * & Y_i & 0 \\
* & * & * & * & I
\end{bmatrix} \succ 0, \\
\end{align*}
\begin{align*}
\text{and } \;\; 
\begin{bmatrix}
Z_i & Z_i & 0 & 0 & 0 & 0 \\
* & Y_i & 0 & \hat F_i & 0 & \hat G_i C_i S_i \\
* & * & \frac{\lambda}{\rho^2_i} I & \hat H_i & 0 & \hat K_i C_i S_i \\
* & * & * & Z_i & Z_i & 0 \\
* & * & * & * & Y_i & 0 \\
* & * & * & * & * & I
\end{bmatrix} \succ 0.
\end{align*}
If these conditions are satisfied, one can recover admissible filter matrices $F_i, G_i, H_i, K_i$ by setting
\begin{align}		\label{eq: recovering the filter}
F_i = V_i^{-1} \hat F_i \hat Z_i^{-1} U_i^{-T}, \;\;
G_i = V_i^{-1} \hat G_i, \;\;
H_i = \hat H_i Z_i^{-1} U_i^{-T},
\;\; K_i = \hat K_i
\end{align}
where $U_i, V_i$ are any two nonsingular matrices such that $V_i U_i^T = I - Y_i Z_i^{-1}$.
\end{thm}

\begin{proof}
For simplicity of notation, let us remove the subscript $i$ in the constraints (\ref{eq: optimization filter - constraint H2})-(\ref{eq: optimization filter - constraint Hinf}),
since we are considering the design of the filters individually. Also, define $\bar \lambda = \lambda / \rho^2$.
The condition (\ref{eq: optimization filter - constraint H2}) is satisfied if and only if there exist matrices $W, \tilde P_1$ such that \cite{Skelton98_book} 
\begin{align} \label{eq: H2 MI}
\Tr(W) < \mu, 
\quad
\begin{bmatrix}
W & \tilde C & \tilde D \\
*  & \tilde P_1 & 0 \\
* & * & I
\end{bmatrix}
\succ 0,
\quad
\begin{bmatrix}
\tilde P_1 & \tilde P_1 \tilde A & \tilde P_1 \tilde B \\
* & \tilde P_1 & 0 \\
* & * & I
\end{bmatrix}
\succ 0.
\end{align}
%
For the constraint (\ref{eq: optimization filter - constraint Hinf}), first note that we have equality of the transfer functions
\[
\TF{F}{G C S}{H}{K C S} = \TF{\begin{array}{cc}A_1 & 0 \\ 0 & F\end{array}}
{\begin{array}{c}0 \\ GCS \end{array}}{\begin{array}{cc}0 & H \end{array}}{KCS}
\]
for any matrix $A_1$, in particular for $A_1$ the zero matrix of the same dimensions as $A$.
With this choice, denote
\[
\bar A = \begin{bmatrix} 0 & 0 \\ 0 & F \end{bmatrix}, \;\; \bar B = \begin{bmatrix} 0 \\ GCS \end{bmatrix}, \bar C = \begin{bmatrix} 0 & H \end{bmatrix}, \bar D = KCS.
\]
Then the constraint (\ref{eq: optimization filter - constraint Hinf}) can be rewritten
$
\|\bar C(sI - \bar A)^{-1} \bar B + \bar D \|_\infty < \bar \lambda, 
$
and is satisfied if and only if there exists a matrix $\tilde P_2$, of the same dimensions as $\tilde P_1$, such that \cite{Skelton98_book}
\begin{align} \label{eq: Hinf MI}
\begin{bmatrix}
\tilde P_2 & 0 & \tilde P_2 \bar A & \tilde P_2 \bar B \\
* & \bar \lambda I & \bar C & \bar D \\
* & * & \tilde P_2 & 0 \\
* & * & * & I
\end{bmatrix} \succ 0.
\end{align}
The sufficient condition of the theorem is obtained by adding the constraint 
\begin{align}	\label{eq: Lyapunov shaping}
\tilde P := \tilde P_1 = \tilde P_2
\end{align}
and using the change of variable suggested in \cite[p. 902]{Scherer97_multiObj}.
Namely, assume that there are matrices $F, G, H, K, \tilde P$, and $W$ satisfying (\ref{eq: H2 MI}), (\ref{eq: Hinf MI}), (\ref{eq: Lyapunov shaping}).
We partition the positive definite matrix $\tilde P$ and its inverse as
\[
\tilde P = \begin{bmatrix}
Y & V \\ V^T & \hat Y
\end{bmatrix},
\quad 
\tilde P^{-1} =  \begin{bmatrix}
X & U \\ U^T & \hat X
\end{bmatrix}.
\]
Note that $YX + VU^T = I$. Define 
\begin{align}		\label{eq: J matrices definitions}
J_1 = \begin{bmatrix}
X & I \\ U^T & 0
\end{bmatrix}, 
\quad
J_2 = \begin{bmatrix}
I & Y \\ 0 & V^T
\end{bmatrix}.
\end{align}
Then we have $\tilde P J_1 = J_2$. Moreover
\begin{align*}
&J_1^T \tilde P J_1 = 
\begin{bmatrix}
X & I \\ I & Y
\end{bmatrix}, \quad
J_1^T \tilde P \tilde A J_1 = 
\begin{bmatrix}
AX & A \\ Y A X+V G C X + V F U^T & Y A + V G C
\end{bmatrix}, \\
&J_1^T \tilde P \tilde B = \begin{bmatrix} B \\ Y B + V G D \end{bmatrix}, 
\quad
\tilde C J_1 = \begin{bmatrix} (L-KC) X - H U^T & L- KC \end{bmatrix}.
\end{align*}
Similarly,
\begin{align*}
J_1^T \tilde P \bar A J_1 = 
\begin{bmatrix}
0 & 0 \\ V F U^T & 0
\end{bmatrix}, 
\quad 
J_1^T \tilde P \bar B = \begin{bmatrix} 0 \\ V G C S \end{bmatrix}, 
\quad
\bar C J_1 = \begin{bmatrix} H U^T & 0 \end{bmatrix}.
\end{align*}
Let $Z = X^{-1}$. Consider first the congruence transformations
\begin{itemize}
\item of the first LMI in (\ref{eq: H2 MI}) by $\text{diag}(I, J_1, I)$ and then by $\text{diag}(I, Z, I, I)$,
\item of the second LMI in (\ref{eq: H2 MI}) by $\text{diag}(J_1, J_1, I)$, and then by $\text{diag}(Z, I, Z, I, I)$,
\item and of the LMI (\ref{eq: Hinf MI}) by $\text{diag}(J_1, I, J_1, I)$, and then by $\text{diag}(Z, I, I, Z, I, I)$.
\end{itemize}
Then, the transformation
$
\hat F = VF U^T Z,
\hat G = VG,
\hat H = H U^T Z,
$
between the filter matrix variables $F, G, H$ and the new variables $\hat F, \hat G, \hat H$ leads to the LMIs of the theorem.
Hence these LMIs are necessarily satisfied if the constraints (\ref{eq: H2 MI}), (\ref{eq: Hinf MI}) are satisfied together with (\ref{eq: Lyapunov shaping}). 

Now suppose that the LMIs of the theorem are satisfied. 
Since $Z \succ 0$, we can define $X = Z^{-1}$. Moreover, since
$\begin{bmatrix}Z & Z \\ Z & Y\end{bmatrix} \succ 0$, we have $Y \succ X^{-1}$
by taking the Schur complement, and so $I-XY$ is nonsingular.
Hence we can find two $n \times n$ nonsingular matrices $U, V$ such that $UV^T = I-XY$.
Then define the nonsingular matrices $J_1, J_2$ as in (\ref{eq: J matrices definitions}), let
$\tilde P = J_2 J_1^{-1}$, and define the matrices $F, G, H, K$ as in (\ref{eq: recovering the filter}).
Since $J_1$ is nonsingular, we can then reverse the congruence transformations to recover 
(\ref{eq: H2 MI}), (\ref{eq: Hinf MI}), which shows that the constraints 
(\ref{eq: optimization filter - constraint H2}), (\ref{eq: optimization filter - constraint Hinf})
are satisfied.
\end{proof}

Note that the problem (\ref{eq: optimization filter}) is also linear in $\mu_i, \lambda$. These variables can then
be minimized subject to the LMI constraints of Theorem \ref{thm: LMI constraints filter design} in order to design
a good filter trading off estimation error and $\ell^2$-sensitivity to minimize the overall MSE. However, including these variables 
directly in the optimization problem can lead to ill-conditioning in the inversion of the matrices $U_i, V_i$ 
in (\ref{eq: recovering the filter}), a phenomenon discussed together with a recommended fix in \cite[p. 903]{Scherer97_multiObj}.

\subsection{Unstable Systems}		\label{section: filter redesign - unstable}

If the dynamics (\ref{eq: linear dynamics participant i}) are not stable, the linear filter design approach
presented in the previous paragraph is not valid. To handle this case, we can further restrict the class of filters.
As before we minimize the estimation error variance together with the sensitivity measured 
by the $\mathcal H_\infty$ norm of the filter.
Starting from the general linear filter dynamics (\ref{eq: filter state}), (\ref{eq: filter output}), we can consider 
designs where $\hat x_i$ is an estimate of $x_i$, and set $H_i = L_i, K_i = 0,$ so that $\hat z_i = L_i \hat x_i$ is an estimate
of $z_i = L_i x_i$. The error dynamics $e_i := x_i - \hat x_i$ then satisfies
\[
e_{i,t+1} = (A_i-G_iC_i) x_{i,t} - F_i \hat x_{i,t} + (B_i - G_i D_i) w_{i,t}.
\]
Setting $F_i = (A_i - G_i C_i)$ gives an error dynamics independent of $x_i$
\begin{equation}	\label{eq: error dynamics - unstable case} 
e_{i,t+1} = (A_i-G_iC_i) e_{i,t} + (B_i - G_i D_i) w_{i,t},
\end{equation}
and leaves the matrix $G_i$ as the only remaining design variable. Note however that the resulting class of filters
contains the (one-step delayed) Kalman filter. To obtain a bounded error, there is an implicit constraint on $G_i$ that $A_i - G_i C_i$ 
should be stable. 

Now, following the discussion in the previous subsection, minimizing the MSE while enforcing differential privacy 
leads to the following optimization problem
\begin{align}
&\min_{\mu_i, \lambda, G_i} \quad \sum_{i=1}^n \mu_i + \kappa(\delta,\epsilon)^2 \lambda \label{eq: optimization filter - unstable} \\
&\text{s.t. } \; \forall \; 1 \leq i \leq n, \;\; 
\| L_i (zI - (A_i - G_i C_i))^{-1} (B_i - G_i D_i) \|_2^2 \leq \mu_i, \label{eq: optimization filter - constraint H2 - unstable} \\
& \rho_i^2 \| L_i (zI - (A_i - G_i C_i))^{-1} G_i C_i S_i \|^2_\infty \leq \lambda. \label{eq: optimization filter - constraint Hinf - unstable}
\end{align}
Again, one can efficiently check a sufficient condition, in the form of the LMIs of the following theorem, 
guaranteeing that the constraints (\ref{eq: optimization filter - constraint H2 - unstable}), (\ref{eq: optimization filter - constraint Hinf - unstable}) 
are satisfied. Optimizing over the variables $\lambda_i, \mu_i, G_i$ can then be done using semidefinite programming.

\begin{thm}	\label{thm: LMI constraints filter design - unstable}
The constraints (\ref{eq: optimization filter - constraint H2 - unstable})-(\ref{eq: optimization filter - constraint Hinf - unstable}), for some $1 \leq i \leq n$,
are satisfied if there exists matrices $Y_i, X_i, \hat G_i$ such that
\begin{align}	\label{H2 unstable LMI}
\Tr(Y_i L_i^T L_i) < \mu_i, \;\; 
\begin{bmatrix}
Y_i & I \\ I & X_i
\end{bmatrix} \succ 0, \;\; 
\begin{bmatrix}
X_i & X_i A_i - \hat G_i C_i & X_i B_i - \hat G_i D_i \\
* & X_i & 0 \\
* & * & I
\end{bmatrix} \succ 0, 
\end{align}
\begin{align}			\label{Hinf unstable LMI}
\text{and }
\begin{bmatrix}
X_i & 0 & X_i A_i - \hat G_i C_i & \hat G_i C_i S_i  \\ 
* & \frac{\lambda}{\rho_i^2} I & L_i & 0 \\
* & * & X_i & 0 \\
* & * & * & I
\end{bmatrix} \succ 0.
\end{align}
If these conditions are satisfied, one can recover an admissible filter matrix $G_i$ by setting
\begin{align*} %
G_i = X_i^{-1} \hat G_i.
\end{align*}
\end{thm}

\begin{proof}
As in Theorem (\ref{thm: LMI constraints filter design}), we simplify the notation below by omitting the subscript $i$.
First, from the error dynamics (\ref{eq: error dynamics - unstable case}), the constraint (\ref{eq: optimization filter - constraint H2 - unstable}) 
is satisfied if and only if there exists a positive definite matrix $P$ such that \cite{Skelton98_book}
\begin{align*}
\Tr(P L^T L) < \mu, \;\; 
(A_i-G_iC_i) P (A_i-G_iC_i)^T + (B_i - G_i D_i) (B_i - G_i D_i)^T \prec P.
\end{align*}
Letting $X = P^{-1}$, introducing the slack variable $Y$, the change of variable $\hat G = XG$, 
and using the Schur complement shows that these conditions are equivalent to the existence 
of two positive definite matrices $X, Y$ such that (\ref{H2 unstable LMI}) is satisfied.
The LMI (\ref{Hinf unstable LMI}) derived from (\ref{eq: optimization filter - constraint Hinf - unstable})
is standard \cite{Skelton98_book}, see also (\ref{eq: Hinf MI}). As in Theorem \ref{thm: LMI constraints filter design}, 
we restrict the search in this LMI to the same matrix $X$ as in (\ref{H2 unstable LMI}), 
which results in a convex problem but introduces some conservatism.
\end{proof}


\section{A Traffic Monitoring Example}	\label{section: traffic monitoring example}

Consider a simplified description of a traffic monitoring system, inspired by real-world implementations
and associated privacy concerns as discussed in \cite{Sun04_trafficEstimation, Hoh11_VTL_trafficMonitoring} 
for example. There are $n$ participating vehicles traveling on a straight road segment. Vehicle $i$, for $1 \leq i \leq n$, is represented
by its state $x_{i,t} = [\xi_{i,t}, \dot \xi_{i,t}]^T$, with $\xi_i$ and $\dot \xi_i$ its position and velocity respectively.
This state evolves as a second-order system with unknown random acceleration inputs
\[
x_{i,t+1} = 
\begin{bmatrix}
1 & T_s \\ 
0 & 1
\end{bmatrix}
x_{i,t}
+ \sigma_{i1}
\begin{bmatrix}
T_s^2 / 2 & 0 \\ T_s & 0
\end{bmatrix}
w_{i,t}, 
\]
where $T_s$ is the sampling period, $w_{i,t}$ is a standard white Gaussian noise, and $\sigma_{i1} > 0$. 
Assume for simplicity that the noise signals $w_j$ for different vehicles are independent.
The traffic monitoring service collects GPS measurements from the 
vehicles \cite{Hoh11_VTL_trafficMonitoring}, i.e., receives noisy readings 
of the positions at the sampling times
\[
y_{i,t} = 
\begin{bmatrix}
1 & 0 
\end{bmatrix}
x_{i,t} 
+ 
\sigma_{i2}
\begin{bmatrix}
0 & 1 
\end{bmatrix}
w_{i,t},
\]
with $\sigma_{i2} > 0$.

The purpose of the traffic monitoring service is to continuously provide an estimate of the 
traffic flow velocity on the road segment, which is approximated by releasing 
at each sampling period an estimate of the average velocity of the participating vehicles, i.e., 
of the quantity
\begin{align}	\label{eq: sample average velocity}
z_t = \frac{1}{n} \sum_{i=1}^n \dot \xi_{i,t}.
\end{align}
With a larger number of participating vehicles, the sample average (\ref{eq: sample average velocity})
represents the traffic flow velocity more accurately. However, while individuals are generally interested 
in the aggregate information provided by such a system, e.g., to estimate their commute time, 
they do not wish their individual trajectories to be publicly revealed, since these might contain 
sensitive information about their driving behavior, frequently visited
locations, etc. 
Privacy-preserving mechanisms for such location-based services are often based on ad-hoc temporal and
spatial cloaking of the measurements \cite{Gruteser03_kAnonymLBS, Hoh11_VTL_trafficMonitoring}. 
However, in the absence of a quantitative definition of privacy and a clear model of the adversary capabilities, 
it is common that proposed techniques are later argued to be deficient \cite{Shokri09_LBSmetric, Shokri10_kAnonFails}. 
The temporal cloaking scheme proposed in \cite{Hoh11_VTL_trafficMonitoring} for example aggregates 
the speed measurements of $k$ users successively crossing a given line, but does not necessarily 
protect individual trajectories against adversaries exploiting temporal relationships between these
aggregated measurements \cite{Shokri09_LBSmetric}.



\subsubsection{Numerical Example} 

We now discuss some differentially private estimators introduced in Section \ref{section: private KF}, in the context of this example.
All individual systems are identical, hence we drop the subscript $i$ in the notation. Assume that the
selection matrix is  $S = \begin{bmatrix} 1 & 0 \\ 0 & 0 \end{bmatrix}$, that $\rho = 100$ m, $T_s=1 s$,
$\sigma_{i1} = \sigma_{i2} =  1$, and $\epsilon = \ln 3$, $\delta = 0.05$. A single Kalman filter 
denoted $\mathcal K$ is designed to provide an estimate $\hat x_i$ of each state vector $x_i$, so that 
in absence of privacy constraint the final estimate would be 
\[
\hat z = \begin{bmatrix} 0 & \frac{1}{n} \end{bmatrix} \sum_{i=1}^n \mathcal K y_i = 
\begin{bmatrix} 0 & 1 \end{bmatrix} \mathcal K \left( \frac{1}{n} \sum_{i=1}^n y_i \right).
\]
Finally, assume that we have $n = 200$ participants, and that their mean initial velocity is $45$ km/h.

In this case, the input noise injection scheme without modification of the Kalman filter is essentially unusable
since its steady-state Root-Mean-Square-Error (RMSE) is almost $26$ km/h. However, modifying the Kalman filter to
take the privacy preserving noise into account as additional measurement noise leads to the best RMSE
of all the schemes discussed here, of about $0.31$ km/h.
Using the Kalman filter $\mathcal K$ with the output noise injection scheme leads to an RMSE of $2.41$ km/h. Moreover
in this case $\| \mathcal K \|_\infty = 0.57$ is quite small, and trying to balance estimation with sensitivity using the LMI
of Theorem \ref{thm: LMI constraints filter design - unstable} (by minimizing the MSE while constraining 
the $\mathcal H_\infty$ norm rather than using the objective function (\ref{eq: optimization filter - unstable}))
only allowed us to reduce this RMSE to $2.31$ km/h. 
However, an issue that is not captured in these steady-state estimation error measures is that of convergence time
of the filters. This is illustrated on Fig. \ref{fig: velocity estimation}, which shows a trajectory of the average 
velocity of the participants, together with the estimates produced by the input noise injection scheme with
compensating Kalman filter and the output noise injection scheme following $\mathcal K$. Although the steady-state RMSE
of the first scheme is much better, its convergence time of more than $1$ min, due to the large privacy-preserving noise, 
is also much larger. This can make this scheme impractical, e.g., if the system is supposed to respond quickly to an abrupt
change in average velocity.

\begin{figure}
\centering
\includegraphics[height=6cm]{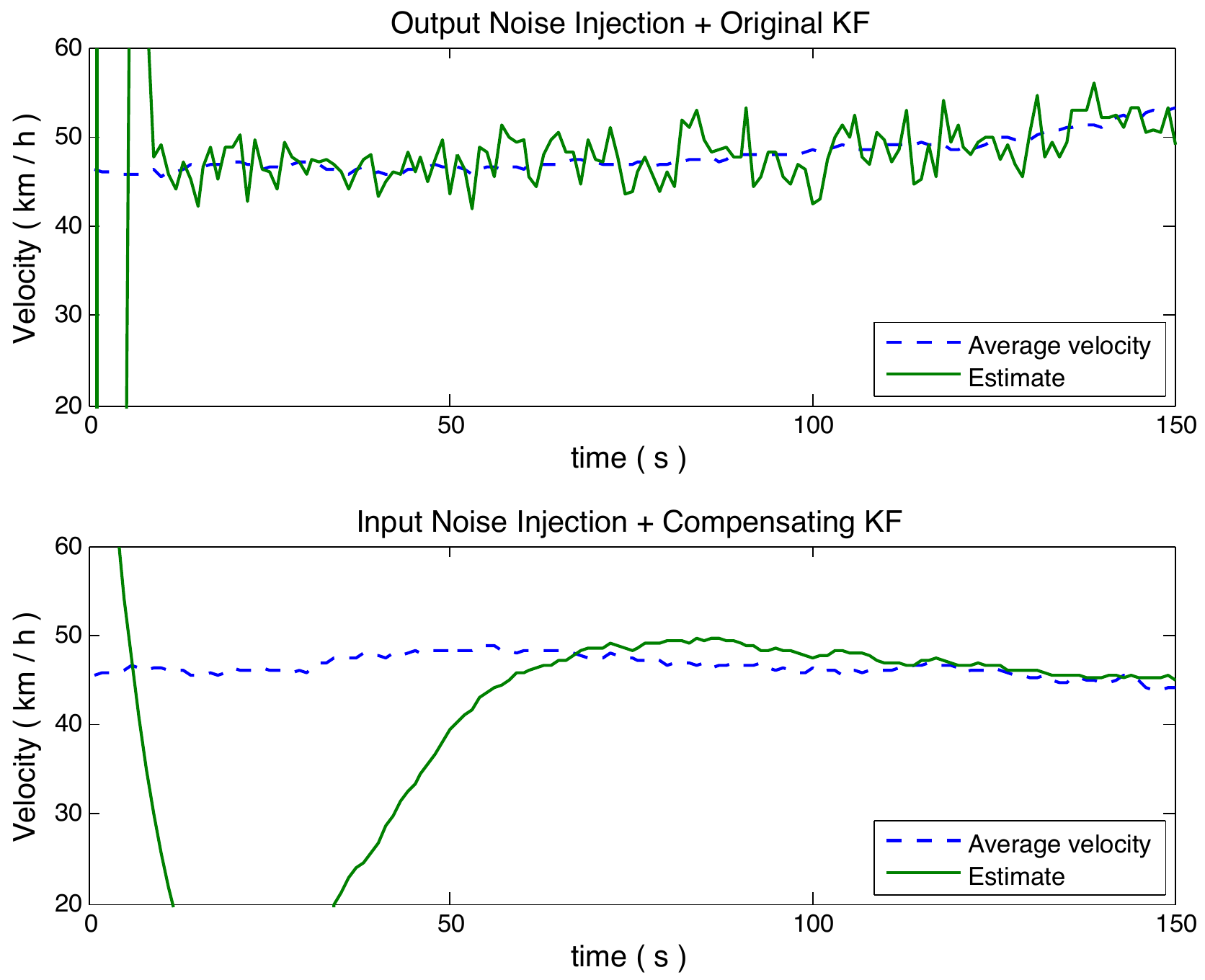}
\caption{Two differentially private average velocity estimates, with $n = 200$ users. 
The Kalman filters are initialized with the same incorrect initial mean velocity ($75$ km/h), in order to
illustrate their convergence time.}
\label{fig: velocity estimation}
\end{figure}


\section{Filtering Event Streams}	\label{section: binary streams}

This section considers an application scenario 
motivated by the work of \cite{Dwork10_DPcounter, Chan10_counter}.
Assume now that an input signal is integer valued, i.e., $u_t \in \mathbb Z$ for all $t \geq 0$. 
Such a signal can record the occurrences of events of interest over time, e.g., the number
of transactions on a commercial website, or the number of people newly infected with a virus. 
As in \cite{Dwork10_DPcounter, Chan10_counter}, two signals $u$ and $u'$
are adjacent if and only if they differ by one at a single time, or equivalently
\begin{equation}	\label{eq: adjacency event-level DP}
\Adj(u,u') \text{ iff } \|u-u'\|_1 = 1.
\end{equation}
The motivation for this adjacency relation is that a given individual contributes
a single event to the stream, and we want to preserve \emph{event-level privacy} \cite{Dwork10_DPcounter},
that is, hide to some extent the presence or absence of an event at a particular time.
This could for example prevent the inference of individual transactions from publicly
available collaborative filtering outputs, as in \cite{Calandrino11_privacyAttackCollabFilt}.
Even though individual events should be hidden, we are still interested in producing
approximate filtered versions of the original signal, e.g., a privacy-preserving moving average
of the input tracking the frequency of events. 
The papers \cite{Dwork10_DPcounter, Chan10_counter} consider specifically the
design of a private counter or accumulator, i.e., a system producing an output signal $y$ 
with $y_t = y_{t-1} + u_t$, where $u$ is binary valued. Note that this system is unstable.
A number of other filters with slowly and monotonically decreasing impulse responses 
are considered in \cite{Bolot11_DPdecayingSums}, 
using a technique similar to \cite{Chan10_counter} based on binary trees. 
Here we show certain approximations of a general linear stable filter $\mathcal G$ 
that preserve event-level privacy. We first make the following remark.

\begin{lem}
Let $\mathcal G$ be a single-input single-output linear system with impulse response $g$. 
Then for the adjacency relation (\ref{eq: adjacency event-level DP})
on integer-valued input signals, the $\ell_p$ sensitivity of $\mathcal G$ is
$\Delta_p \mathcal G = \|g\|_p$. In particular for $p=2$, we have
$\Delta_2 \mathcal G = \|\mathcal G\|_2$, the $\mathcal H_2$ norm of $\mathcal G$.
\end{lem}

\begin{proof}
For two adjacent binary-valued signals $u, u'$, 
we have that $u-u'$ is a positive or negative 
impulse signal $\delta$, and hence
\ifthenelse {\boolean{TwoColEq}}
{
\begin{align*}
\|\mathcal Gu - \mathcal Gu'\|_p &= \|\mathcal G(u-u')\|_p = \|\mathcal G \delta\|_p = \|g * \delta\|_p \\
&= \|g\|_p.
\end{align*}
}
{
\begin{align*}
\|\mathcal Gu - \mathcal Gu'\|_p &= \|\mathcal G(u-u')\|_p = \|\mathcal G \delta\|_p = \|g * \delta\|_p = \|g\|_p.
\end{align*}
}
\end{proof}

We measure the utility of specific schemes throughout this section by
the MSE between the published and desired outputs.
Similarly to our discussion at the end of Section \ref{section: DP linear systems},
there are two straightforward mechanisms that provide differential privacy.
One can add white noise $w$ directly on the input signal, with
$w_t \sim \Lap(1/\epsilon)$ for the Laplace mechanism and
$w_t \sim \mathcal N(0,\kappa(\delta,\epsilon))$ for the Gaussian mechanism.
Or one can add noise at the output of the filter $\mathcal G$, with $w_t \sim \Lap(\|g\|_1/\epsilon)$
for the Laplace mechanism and $w_t \sim \mathcal N(0,\|g\|_2 \kappa(\delta,\epsilon))$ for the
Gaussian mechanism. 
For the Gaussian mechanism, one obtains in both cases an MSE 
equal to $\|\mathcal G\|^2_2 \; \kappa(\delta,\epsilon)^2$.
For the Laplace mechanism, it is always better 
to add the noise at the input. Indeed, we obtain in this case an MSE 
of $2 \|g\|_2^2 / \epsilon^2$ instead of the greater $2 \|g\|_1^2 / \epsilon^2$
if the noise is added at the output.

\begin{figure}
\centering
\includegraphics[height=3cm]{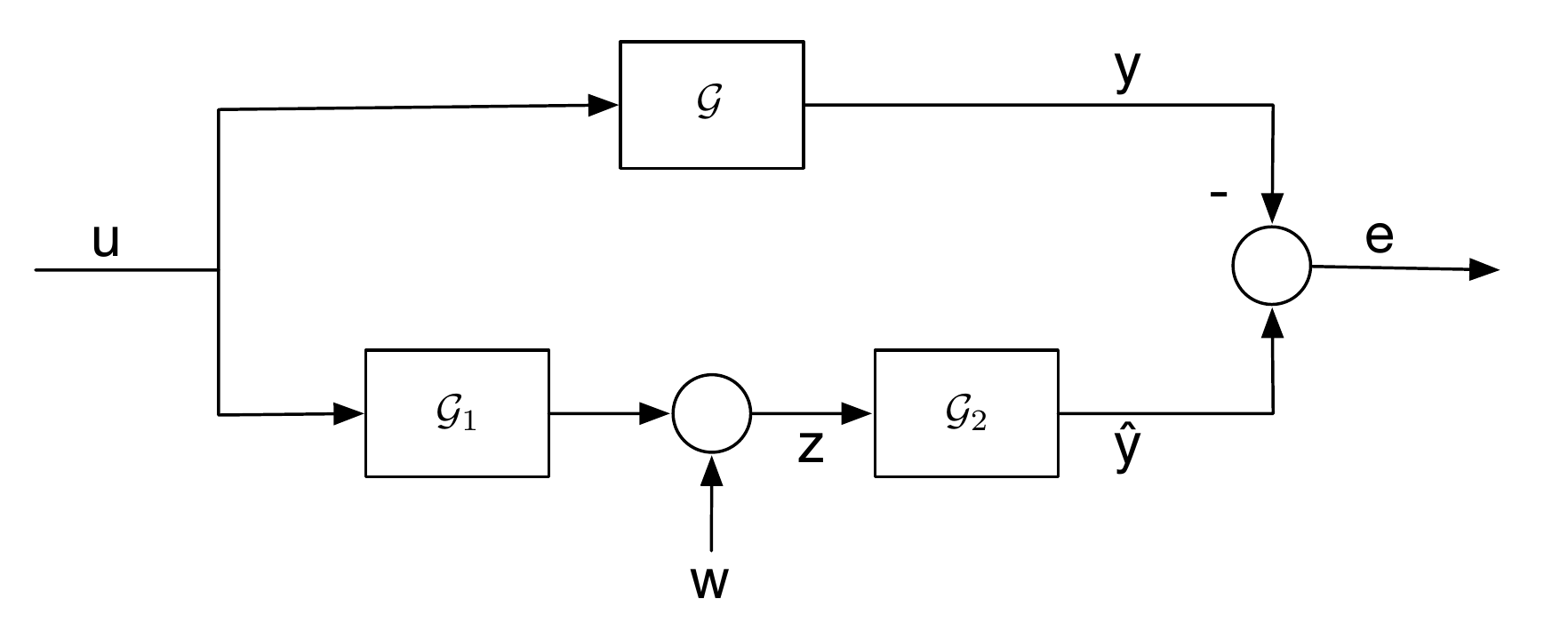}
\caption{Differentially private filter approximation set-up.}
\label{fig: DP filter approximation}
\end{figure}

We now generalize these mechanisms to the approximation set-up shown on Fig. \ref{fig: DP filter approximation}. 
The previous mechanisms are recovered when $\mathcal G_1$ or $\mathcal G_2$ is the identity operator.
To show that one can improve the utility of the mechanism with this set-up, consider the following choice of filters 
$\mathcal G_1$ and $\mathcal G_2$. Let $\mathcal G_1$ be a stable, minimum phase filter (hence invertible). 
Let $\mathcal G_2 = \mathcal G \mathcal G_1^{-1}$. We call this particular
choice the \emph{zero forcing equalization} (ZFE) mechanism. 
To guarantee $(\epsilon,\delta)$-differential privacy,
the noise $w$ is chosen to be white Gaussian 
with $\sigma = \kappa(\delta,\epsilon) \|\mathcal G_1\|_2$.
The MSE for the ZFE mechanism is 
\ifthenelse {\boolean{TwoColEq}} 
{
\begin{align*}
e^{ZFE}_{mse} := 
&\lim_{T \to \infty} \frac{1}{T} \sum_{t=0}^\infty \Exp[\| (\mathcal Gu)_t - (\mathcal Gu + \mathcal G \mathcal G_1^{-1}w)_t\|^2_2] \\
=& \lim_{T \to \infty} \frac{1}{T} \sum_{t=0}^\infty \Exp[\|(\mathcal G \mathcal G_1^{-1}w)_t\|_2^2] \\
=& \kappa(\epsilon, \delta)^2 \|\mathcal G_1\|_2^2 \|\mathcal G \mathcal G_1^{-1}\|^2_2.
\end{align*}
}
{
\begin{align*}
e^{ZFE}_{mse} := 
&\lim_{T \to \infty} \frac{1}{T} \sum_{t=0}^\infty \Exp[\| (\mathcal Gu)_t - (\mathcal Gu + \mathcal G \mathcal G_1^{-1}w)_t\|^2_2] \\
=& \lim_{T \to \infty} \frac{1}{T} \sum_{t=0}^\infty \Exp[\|(\mathcal G \mathcal G_1^{-1}w)_t\|_2^2] 
= \kappa(\epsilon, \delta)^2 \|\mathcal G_1\|_2^2 \|\mathcal G \mathcal G_1^{-1}\|^2_2.
\end{align*}
}
Hence we are lead to consider the following problem
\ifthenelse {\boolean{TwoColEq}} 
{
\begin{align*}
&\min_{\mathcal G_1} \|\mathcal G_1\|_2^2 \|\mathcal G _\mathcal G_1^{-1}\|^2_2 \\
&= \min_{\mathcal G_1} \frac{1}{4 \pi^2} \int_{-\pi}^\pi |\mathcal G_1(e^{j\omega})|^2 d \omega 
\int_{-\pi}^\pi \left| \frac{\mathcal G(e^{j\omega})}{\mathcal G_1(e^{j \omega})} \right|^2 d \omega,
\end{align*}
}
{
\begin{align*}
\min_{\mathcal G_1} \|\mathcal G_1\|_2^2 \|\mathcal G \mathcal G_1^{-1}\|^2_2 
= \min_{\mathcal G_1} \frac{1}{4 \pi^2} \int_{-\pi}^\pi |\mathcal G_1(e^{j\omega})|^2 d \omega 
\int_{-\pi}^\pi \left| \frac{\mathcal G(e^{j\omega})}{\mathcal G_1(e^{j \omega})} \right|^2 d \omega,
\end{align*}
}
where the minimization is over the stable, minimum phase transfer functions $\mathcal G_1$. 

\begin{thm}	\label{thm: error for ZFE mechanism}
We have, for any stable, minimum phase system $\mathcal G_1$, 
\[
e^{ZFE}_{mse} \geq \kappa(\epsilon, \delta)^2 \left(\frac{1}{2 \pi} \int_{-\pi}^\pi |\mathcal G(e^{j\omega})| d \omega \right)^2.
\]
This lower bound on the mean-squared error of the ZFE mechanism is
attained by letting $|\mathcal G_1(e^{j \omega})|^2 = \lambda |\mathcal G(e^{j \omega})|$ for all $\omega \in [-\pi,\pi)$,
where $\lambda$ is some arbitrary positive number.
It can be approached arbitrarily closely by stable, rational, 
minimum phase transfer functions $\mathcal G_1$.
\end{thm}

\begin{proof}
By the Cauchy-Schwarz inequality, we have
\ifthenelse {\boolean{TwoColEq}} 
{
\begin{align*}
&\left( \int_{-\pi}^\pi |\mathcal G(e^{j\omega})| d \omega \right)^2 =
\left(\int_{-\pi}^\pi  |\mathcal G_1(e^{j\omega})| \left| \frac{\mathcal G(e^{j\omega})}{\mathcal G_1(e^{j \omega})} \right| d \omega \right)^2 \\
&\leq \int_{-\pi}^\pi |\mathcal G_1(e^{j\omega})|^2 d \omega 
\int_{-\pi}^\pi \left| \frac{\mathcal G(e^{j\omega})}{\mathcal G_1(e^{j \omega})} \right|^2 d \omega,
\end{align*}
}
{
\begin{align*}
\left( \int_{-\pi}^\pi |\mathcal G(e^{j\omega})| d \omega \right)^2 =
\left(\int_{-\pi}^\pi  |\mathcal G_1(e^{j\omega})| \left| \frac{\mathcal G(e^{j\omega})}{\mathcal G_1(e^{j \omega})} \right| d \omega \right)^2 
\leq \int_{-\pi}^\pi |\mathcal G_1(e^{j\omega})|^2 d \omega 
\int_{-\pi}^\pi \left| \frac{\mathcal G(e^{j\omega})}{\mathcal G_1(e^{j \omega})} \right|^2 d \omega,
\end{align*}

}
hence the bound. Moreover, equality is attained if and only if there exists $\lambda \in \mathbb R$
such that
\ifthenelse {\boolean{TwoColEq}} 
{
\begin{align*}
& |\mathcal G_1(e^{j\omega})| = \lambda \left| \frac{\mathcal G(e^{j\omega})}{\mathcal G_1(e^{j \omega})} \right|, \\
\text{i.e., } & |\mathcal G_1(e^{j\omega})|^2 = \lambda |\mathcal G(e^{j\omega})|, \;\; \forall \omega \in \mathbb R.
\end{align*}
}
{
\begin{align*}
|\mathcal G_1(e^{j\omega})| = \lambda \left| \frac{\mathcal G(e^{j\omega})}{\mathcal G_1(e^{j \omega})} \right|, 
\text{i.e., } |\mathcal G_1(e^{j\omega})|^2 = \lambda |\mathcal G(e^{j\omega})|, \;\; \forall \omega \in \mathbb R.
\end{align*}
}
To see that the bound can be approached using finite-dimensional filters,
by Weierstrass theorem we can first approximate $|\mathcal G(e^{j \omega})|$ 
arbitrarily closely by a rational positive function $\hat{\mathcal G}$. 
We then set $\mathcal G_1$ to be the minimum-phase spectral factor of $\hat{\mathcal G}$.
\end{proof}

The MSE obtained for the best ZFE mechanism 
in Theorem \ref{thm: error for ZFE mechanism} cannot be worse than
the MSE for the scheme adding noise at the input, and is generally strictly
smaller, since by Jensen's inequality we have
\[
\left(\int_{-\pi}^\pi |\mathcal G(e^{j\omega})| \frac{d \omega}{2 \pi}  \right)^2
\leq \int_{-\pi}^\pi |\mathcal G(e^{j\omega})|^2 \frac{d \omega}{2 \pi} = \|\mathcal G\|_2^2.
\]
In addition, the MSE of the ZFE mechanism is independent of the input signal $u$. 
However, a smaller error could be obtained with other schemes, 
in particular schemes that exploit some knowledge about the input signal. 
Note that once $\mathcal G_1$ is chosen, designing $\mathcal G_2$ is a standard 
equalization problem \cite{Proakis00_digitalComBook}. The name of the
ZFE mechanism is motivated by the choice of trying to cancel the effect of 
$\mathcal G_1$ by using its inverse (zero forcing equalizer).
Nonlinear components can be very useful as well. In particular if we add the hypothesis
that the input signal is binary valued, as in \cite{Dwork10_DPcounter, Chan10_counter}, 
we can modify the simple scheme adding noise at the input by including a detector $H$ in front of the system $\mathcal G$, 
namely, for $\hat u_t = u_t + w_t$,
 \[
 H(\hat u_t) = \begin{cases}
 1, \;\; \hat u_t \geq 1/2, \\
 0, \;\; \hat u_t < 1/2.
 \end{cases}
 \]
 This exploits the knowledge that the input signal is binary valued, preserves differential privacy
 by Theorem \ref{thm: resilience to post-processing}, and sometimes significantly improves
 the MSE, depending on other characteristics of the signal.
 

\subsection{Exploiting Additional Public Knowledge}

To further illustrate the idea of exploiting potentially available additional knowledge about the input signal, consider using
a minimum mean squared error (MMSE) estimator for $\mathcal G_2$ rather than employing $ \mathcal G \mathcal G_1^{-1}$, 
since the latter can significantly amplify the noise at frequencies where $\mathcal G_1$ is small. 
Let us assume that $\mathcal G_1$ is already chosen, e.g., according to Theorem \ref{thm: error for ZFE mechanism}
(this choice is not optimal any more if $\mathcal G_2$ is not $\mathcal G\mathcal G_1^{-1}$). 
%
%
Moreover, assume that that it is publicly known that $u$ is wide-sense stationary 
with mean and autocorrelation denoted
\[
\Exp[u_t] = \mu, \;\; \Exp[u_s u_t] =: R_u[s-t].
\]
From this data, the second order statistics of $y$ and $z$ on Fig. $1$ are also known, 
in particular
\[
R_z = f * \tilde f * R_u + \sigma^2 \delta,  R_{yz} = g * \tilde f * R_u,
\]
where $\sigma^2 = \kappa(\delta,\epsilon)^2 \| \mathcal G_1\|_2^2$, $\delta$ is the impulse signal,
$f$ is the impulse response of $\mathcal G_1$, 
and $\tilde f_t = f_{-t}$.
We then design $\mathcal G_2$ to minimize the MSE
\[
\Exp[|y_t - \hat y_t|^2].
\]
For simplicity, consider the case where $\mathcal G_2$ is restricted to be a finite-impulse response filter, i.e.,
\[
\hat y_t = (\mathcal G_2 z)_t = \sum_{k=0}^N h_k z_{t-k},
\]
where $N$ is the order of the filter. The vector $h=[h_0,\ldots,h_N]^T$ is the solution of the Yule-Walker equations \cite{Poor94_SPbook}
\[
\begin{bmatrix}
R_z[0] & R_z[1] & \ldots & R_z[N]    \\
R_z[1] & R_z[0] & \ldots & R_z[N-1] \\
\vdots & \vdots & \vdots & \vdots \\
R_z[N] & \ldots & \ldots & R_z[0]
\end{bmatrix}
h
=
\begin{bmatrix}
R_{yz}[0] \\
\vdots \\
R_{yz}[N]
\end{bmatrix}
\]
According to Theorem \ref{thm: resilience to post-processing}, 
differential privacy is preserved since the filter $\mathcal G_2$ only processes 
the already differentially private signal $z$.
Even if the statistical assumptions turn out not to be satisfied
by $u$, the privacy guarantee still holds and only performance
is impacted.

\begin{exmp}
Fig. \ref{fig: binary filtering illustration} illustrates the differentially private
output obtained by the MMSE mechanism approximating the filter
$\mathcal G = 1/(s(z)+0.05)$, with $s(z$) the bilinear transformation
\[
s(z) = 2 \frac{1-z^{-1}}{1+z^{-1}}.
\]
The input signal is binary valued and the privacy parameters are set to $\epsilon = \ln 3$, $\delta = 0.05$.
For this specific input, the empirical MSE of the ZFE is $5.8$, 
compared to $4.6$ for the MMSE mechanism.
The simpler scheme with noise added at the input is essentially unusable, since its
MSE is $\kappa(\delta,\epsilon)^2 \|\mathcal G\|_2^2 \approx 30.1$.
Adding a detector reduces this MSE to about $17$.
\begin{figure}
\centering
\includegraphics[width=8cm,height=7cm]{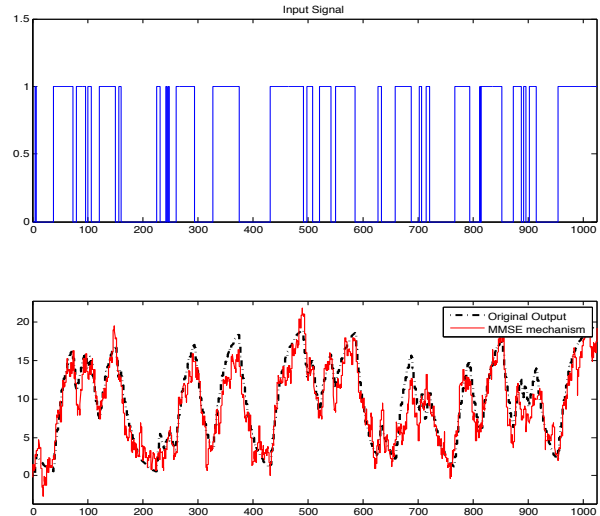}
\caption{Sample path for the MMSE mechanism.}
\label{fig: binary filtering illustration}
\end{figure}
\end{exmp}

\subsection{Related Work}	\label{section: related work}


Some papers closely related to the event filtering problem considered in this section are 
\cite{Dwork10_DPcounter, Rastogi10_DPtimeSeries, Chan11_DPcontinuous, Bolot11_DPdecayingSums}.
As previously mentioned, \cite{Dwork10_DPcounter, Rastogi10_DPtimeSeries} consider an unstable filter, 
the accumulator. The techniques employed there are quite different, relying essentially on binary
trees to keep track of intermediate calculations and reduce the amount of noise introduced 
by the privacy mechanism. Bolot et al. \cite{Bolot11_DPdecayingSums} extend this technique 
to the differentially private approximation of certain filters with monotonic, 
slowly decaying impulse response. In fact, this technique can be extended to general 
linear systems by using a state-space realization and keeping track of the system state at 
carefully chosen times in a binary tree.
However, the usefulness of this approach seems to be limited for most practical stable filters,
the resulting MSE being typically too large and the implementation of the scheme significantly more 
complex than for a simple recursive filter.


Finally, as with the MMSE estimation mechanism, 
one can try to use additional information 
about the input signals to calibrate the amount 
of noise introduced by the privacy mechanism. 
For example, if there exists a sparse representation 
of the signal in some basis (such as a Fourier or a wavelet basis),
then one can try to perturb the representation coefficients in this alternate basis. 
For example, \cite{Rastogi10_DPtimeSeries} perturbs the largest coefficients of the
Discrete Fourier Transform of the signal. 
A difficulty with such approaches 
is that they are typically not causal and not recursive, 
requiring an amount of processing that increases with time.

\section{Conclusion}

We have discussed mechanisms for preserving the differential privacy of individual users
transmitting time-varying signals to a trusted central server releasing sanitized filtered outputs
based on these inputs. Decentralized versions of the mechanism of Section \ref{section: DP linear systems} 
can in fact be implemented in the absence of trusted server by means of cryptographic techniques \cite{Rastogi10_DPtimeSeries}.
We believe that research on privacy issues is critical to encourage the development of future cyber-physical systems, 
which typically rely on the users data to improve their efficiency. 
Numerous directions of study are open for dynamical systems, including designing better filtering mechanisms, 
and understanding design trade-offs between privacy or security and performance in large-scale control systems.


%
%
%
%
%
%
\section*{Acknowledgment}
The authors would like to thank Aaron Roth for providing valuable insight into
the notion of differential privacy.
%
%




\bibliographystyle{IEEEtran/IEEEtran}


\bibliography{IEEEtranBST/IEEEabrv,/Users/jerome/Dropbox/Research/bibtex/energy,/Users/jerome/Dropbox/Research/bibtex/probability,/Users/jerome/Dropbox/Research/bibtex/securityPrivacy,/Users/jerome/Dropbox/Research/bibtex/signalProcessing,/Users/jerome/Dropbox/Research/bibtex/controlSystems,/Users/jerome/Dropbox/Research/bibtex/communications} 

\begin{thebibliography}{10}
\providecommand{\url}[1]{#1}
\csname url@samestyle\endcsname
\providecommand{\newblock}{\relax}
\providecommand{\bibinfo}[2]{#2}
\providecommand{\BIBentrySTDinterwordspacing}{\spaceskip=0pt\relax}
\providecommand{\BIBentryALTinterwordstretchfactor}{4}
\providecommand{\BIBentryALTinterwordspacing}{\spaceskip=\fontdimen2\font plus
\BIBentryALTinterwordstretchfactor\fontdimen3\font minus
  \fontdimen4\font\relax}
\providecommand{\BIBforeignlanguage}[2]{{%
\expandafter\ifx\csname l@#1\endcsname\relax
\typeout{** WARNING: IEEEtran.bst: No hyphenation pattern has been}%
\typeout{** loaded for the language `#1'. Using the pattern for}%
\typeout{** the default language instead.}%
\else
\language=\csname l@#1\endcsname
\fi
#2}}
\providecommand{\BIBdecl}{\relax}
\BIBdecl

\bibitem{Hart92_loadMonitoring}
G.~W. Hart, ``Nonintrusive appliance load monitoring,'' \emph{Proceedings of
  the IEEE}, vol.~80, no.~12, pp. 1870--1891, December 1992.

\bibitem{Narayanan08_netflixBreach}
A.~Narayanan and V.~Shmatikov, ``Robust de-anonymization of large sparse
  datasets (how to break anonymity of the {Netflix Prize} dataset),'' in
  \emph{Proceedings of the 2008 IEEE Symposium on Security and Privacy}, 2008.

\bibitem{Calandrino11_privacyAttackCollabFilt}
J.~A. Calandrino, A.~Kilzer, A.~Narayanan, E.~W. Felten, and V.~Shmatikov,
  ````you might also like'': Privacy risks of collaborative filtering,'' in
  \emph{IEEE Symposium on Security and Privacy}, 2011.

\bibitem{Hoh11_VTL_trafficMonitoring}
B.~Hoh, T.~Iwuchukwu, Q.~Jacobson, M.~Gruteser, A.~Bayen, J.-C. Herrera,
  R.~Herring, D.~Work, M.~Annavaram, and J.~Ban, ``Enhancing privacy and
  accuracy in probe vehicle based traffic monitoring via virtual trip lines,''
  \emph{IEEE Transactions on Mobile Computing}, 2011.

\bibitem{Dwork06_DPcalibration}
C.~Dwork, F.~{McSherry}, K.~Nissim, and A.~Smith, ``Calibrating noise to
  sensitivity in private data analysis,'' in \emph{Proceedings of the Third
  Theory of Cryptography Conference}, 2006, pp. 265--284.

\bibitem{Kasiviswanathan08_DP_sideInfo}
\BIBentryALTinterwordspacing
S.~P. Kasiviswanathan and A.~Smith, ``A note on differential privacy: Defining
  resistance to arbitrary side information,'' March 2008. [Online]. Available:
  \url{http://arxiv.org/abs/0803.3946}
\BIBentrySTDinterwordspacing

\bibitem{Dwork06_DPgaussian}
C.~Dwork, K.~Kenthapadi, F.~McSherry, I.~M.~M. Naor, and Naor, ``Our data,
  ourselves: Privacy via distributed noise generation,'' \emph{Advances in
  Cryptology-EUROCRYPT 2006}, pp. 486--503, 2006.

\bibitem{Roth10_DP_thesis}
A.~Roth, ``New algorithms for preserving differential privacy,'' Ph.D.
  dissertation, Carnegie Mellon University, 2010.

\bibitem{Li10_DPmatrix1}
C.~Li, M.~Hay, V.~Rastogi, G.~Miklau, and A.~{McGregor}, ``Optimizing linear
  counting queries under differential privacy,'' in \emph{Principles of
  Database Systems (PODS)}, 2010.

\bibitem{Dwork10_DPcounter}
C.~Dwork, M.~Naor, T.~Pitassi, and G.~N. Rothblum, ``Differential privacy under
  continual observations,'' in \emph{STOC'10}, Cambridge, MA, June 2010.

\bibitem{Chan11_DPcontinuous}
T.-H.~H. Chan, E.~Shi, and D.~Song, ``Private and continual release of
  statistics,'' \emph{ACM Transactions on Information and System Security},
  vol.~14, no.~3, pp. 26:1--26:24, November 2011.

\bibitem{Bolot11_DPdecayingSums}
J.~Bolot, N.~Fawaz, S.~Muthukrishnan, A.~Nikolov, and N.~Taft, ``Private
  decayed sum estimation under continual observation,'' September 2011,
  {http://arxiv.org/abs/1108.6123}.

\bibitem{Huang12_DPconsensus}
Z.~Huang, S.~Mitra, and G.~Dullerud, ``Differentially private iterative
  synchronous consensus,'' in \emph{Proceedings of the CCS Workshop on Privacy
  in the Electronic Society (WPES)}, Raleigh, North Carolina, October 2012, to
  appear.

\bibitem{Varodayan11_batteryPrivacy}
D.~Varodayan and A.~Khisti, ``Smart meter privacy using a rechargeable battery:
  minimizing the rate of information leakage,'' in \emph{Proceedings of the
  IEEE International Conference on Acoustics, Speech, and Signal Processing},
  Prag, Czech Republic, 2011.

\bibitem{Sankar11_privacyInfoTheoretic}
L.~Sankar, S.~R. Rajagopalan, and H.~V. Poor, ``A theory of privacy and utility
  in databases,'' Princeton University, Tech. Rep., February 2011.

\bibitem{Blum05_sulq}
A.~Blum, C.~Dwork, F.~McSherry, and K.~Nissim, ``Practical privacy: the {SuLQ}
  framework,'' in \emph{Proceedings of the twenty-fourth ACM
  SIGMOD-SIGACT-SIGART symposium on Principles of database systems (PODS)}, New
  York, NY, USA, 2005, pp. 128--138.

\bibitem{Dwork_ICAL06_DP}
C.~Dwork, ``Differential privacy,'' in \emph{Proceedings of the 33rd
  International Colloquium on Automata, Languages and Programming (ICALP)},
  ser. Lecture Notes in Computer Science, vol. 4052.\hskip 1em plus 0.5em minus
  0.4em\relax Springer-Verlag, 2006.

\bibitem{Dudley02_book}
R.~M. Dudley, \emph{Real Analysis and Probability}, 2nd~ed.\hskip 1em plus
  0.5em minus 0.4em\relax Cambridge University Press, 2002.

\bibitem{Cover91_infoTheory}
T.~M. Cover and J.~A. Thomas, \emph{Elements of Information Theory}.\hskip 1em
  plus 0.5em minus 0.4em\relax New York, NY: John Wiley and Sons, 1991.

\bibitem{Breiman92_book}
L.~Breiman, \emph{Probability}, ser. Classics in Applied Mathematics.\hskip 1em
  plus 0.5em minus 0.4em\relax SIAM, 1992.

\bibitem{VanderSchaft00_passivity}
A.~{van der Schaft}, \emph{{L2-gain} and passivity techniques in nonlinear
  control}.\hskip 1em plus 0.5em minus 0.4em\relax Springer Verlag, 2000.

\bibitem{Shi11_DPaggregation}
E.~Shi, T.-H.~H. Chan, E.~Rieffel, R.~Chow, and D.~Song, ``Privacy-preserving
  aggregation of time-series data,'' in \emph{Proceedings of 18th Annual
  Network and Distributed System Security Symposium (NDSS 2011)}, February
  2011.

\bibitem{Anderson05_filtering}
B.~D.~O. Anderson and J.~B. Moore, \emph{Optimal Filtering}.\hskip 1em plus
  0.5em minus 0.4em\relax Dover, 2005.

\bibitem{Skelton98_book}
R.~E. Skelton, T.~Iwasaki, and K.~Grigoriadis, \emph{A Unified Algebraic
  Approach to Linear Control Design}.\hskip 1em plus 0.5em minus 0.4em\relax
  Taylor and Francis, 1998.

\bibitem{Scherer97_multiObj}
C.~Scherer, P.~Gahinet, and M.~Chilali, ``Multiobjective output-feedback
  control via {LMI} optimization,'' \emph{IEEE Transactions on Automatic
  Control}, vol.~42, no.~7, pp. 896--911, July 1997.

\bibitem{Sun04_trafficEstimation}
X.~Sun, L.~Munoz, and R.~Horowitz, ``Mixture {K}alman filter based highway
  congestion mode and vehicle density estimator and its application,'' in
  \emph{Proceedings of the American Control Conference}, July 2004, pp.
  2098--2103.

\bibitem{Gruteser03_kAnonymLBS}
M.~Gruteser and D.~Grunwald, ``Anonymous usage of location-based services
  through spatial and temporal cloaking,'' in \emph{ACM MobiSys}, 2003.

\bibitem{Shokri09_LBSmetric}
R.~Shokri, J.~Freudiger, M.~Jadliwala, and J.-P. Hubaux, ``A distortion-based
  metric for location privacy,'' in \emph{Proceedings of the CCS Workshop on
  Privacy in the Electronic Society (WPES)}, 2009.

\bibitem{Shokri10_kAnonFails}
R.~Shokri, C.~Troncoso, C.~Diaz, J.~Freudiger, and J.-P. Hubaux, ``Unraveling
  an old cloak: $k$-anonymity for location privacy,'' in \emph{ACM Workshop on
  Privacy in the Electronic Society (WPES)}.\hskip 1em plus 0.5em minus
  0.4em\relax ACM, 2010.

\bibitem{Chan10_counter}
T.-H.~H. Chan, E.~Shi, and D.~Song, ``Private and continual release of
  statistics,'' University of California at Berkeley, Tech. Rep., 2010.

\bibitem{Proakis00_digitalComBook}
J.~Proakis, \emph{Digital Communications}.\hskip 1em plus 0.5em minus
  0.4em\relax McGraw-Hill, 2000.

\bibitem{Poor94_SPbook}
H.~V. Poor, \emph{An Introduction to Signal Detection and Estimation},
  2nd~ed.\hskip 1em plus 0.5em minus 0.4em\relax Springer, 1994.

\bibitem{Rastogi10_DPtimeSeries}
V.~Rastogi and S.~Nath, ``Differentially private aggregation of distributed
  time-series with transformation and encryption,'' in \emph{Proceedings of the
  ACM Conference on Management of Data (SIGMOD)}, Indianapolis, IN, June 2010.

\end{thebibliography}

\end{document}